\documentclass[a4paper,11pt]{article}
\usepackage[T1]{fontenc}
\usepackage{amsmath,amsfonts,amsthm,amssymb}
\usepackage{siunitx}
\usepackage{authblk}
\usepackage[hmargin={27mm,27mm},vmargin={32mm,36mm}]{geometry}
\usepackage[latin1]{inputenc}
\usepackage{newtxtext,newtxmath}
\usepackage{numprint} 
\usepackage{paralist}
\usepackage[colorlinks,allcolors={blue},linkcolor={black}]{hyperref}
\usepackage[bold]{hhtensor}
\usepackage{comment}
\usepackage{nicefrac}
\usepackage{enumerate}
\usepackage{graphicx}
\usepackage[font=footnotesize]{caption}
\usepackage[font=footnotesize]{subcaption}
\usepackage{pgfplots}
\usepackage{xcolor}
\usepackage{tikz}
\pgfplotsset{compat=1.16}
\definecolor{myred}{rgb}{0.9, 0.0, 0.0}
\definecolor{myblue}{rgb}{0.0, 0.28, 0.67}
\definecolor{mygreen}{rgb}{0.0, 0.7, 0.0}
\definecolor{myyellow}{rgb}{1.0, 0.55, 0.0}
\definecolor{mypurple}{rgb}{0.5, 0.1, 0.5}

\newcommand{\bvec}[1]{\boldsymbol{#1}}

\newcommand{\btens}[1]{\boldsymbol{#1}}

\newcommand{\GRAD}{\bvec{\nabla}}
\newcommand{\GRADs}{\GRAD_\symm}

\newcommand{\DIV}{\bvec{\nabla}{\cdot}}

\newcommand{\ud}{\,\mathrm{d}}

\newcommand{\Id}[1][d]{\bvec{I}_{#1}}
\newcommand{\norm}[2][]{\|#2\|_{#1}}
\newcommand{\seminorm}[2][]{|#2|_{#1}}

\newcommand{\symm}{{\rm s}}
\newcommand{\sks}{{\rm sk}}

\newcommand{\diff}{\btens{K}}

\newcommand{\Real}{\mathbb{R}}
\newcommand{\Natural}{\mathbb{N}}
\newcommand{\Integer}{\mathbb{Z}}

\newcommand{\Lvec}[2][{2}]{\textbf{L}^{#1}(#2)}
\newcommand{\Lmat}[2][{2}]{\mathbb{L}^{#1}(#2)}
\newcommand{\Lmats}[2][{2}]{\mathbb{L}^{#1}_\symm(#2)}
\newcommand{\Lmatsk}[2][{2}]{\mathbb{L}^{#1}_{\sks}(#2)}

\newcommand{\Hvec}[2][{1}]{\textbf{H}^{#1}(#2)}
\newcommand{\HvecD}[2][{1}]{\textbf{H}_{0,\Gamma_d}^{#1}(#2)}

\newcommand{\Hdiv}[1]{\textbf{H}({\rm div},#1)}
\newcommand{\Hdivmat}[1]{\mathbb{H}({\rm div},#1)}
\newcommand{\Hdivmats}[1]{\mathbb{H}_\symm({\rm div},#1)}

\newcommand{\Th}[1][h]{\mathcal{T}_{#1}}

\newcommand{\normal}{\bvec{n}}

\newcommand{\tF}{t_{\rm F}}

\newcommand{\lproj}[1][h]{\Pi_{#1}}
\newcommand{\vlproj}[1][h]{\bvec{\Pi}_{#1}}

\newtheorem{theorem}{Theorem}

\newtheorem{lemma}[theorem]{Lemma}

\theoremstyle{remark}
\newtheorem{remark}[theorem]{Remark}
\theoremstyle{definition}

\title{A robust fully-mixed finite element method with skew-symmetry penalization for low-frequency poroelasticity}
\author{Stefano Bonetti\thanks{MOX, Department of Mathematics, Politecnico di Milano, 20133 Milano, Italy (stefano.bonetti@polimi.it)},  
Michele Botti\thanks{MOX, Department of Mathematics, Politecnico di Milano, 20133 Milano, Italy (michele.botti@polimi.it)},  and
Patrick Vega\thanks{Departamento de Matematica y Ciencia de la Computaci\'on, Universidad de Santiago de Chile, 9170022 Santiago, Chile (patrick.vega@usach.cl)}}
\date{}

\begin{document}
\maketitle

\begin{abstract}
\noindent
In this work, we present and analyze a fully-mixed finite element scheme for the dynamic poroelasticity problem in the low-frequency regime. We write the problem as a four-field, first-order, hyperbolic system of equations where the symmetry constraint on the stress field is imposed via penalization. 
This strategy is equivalent to adding a perturbation to the saddle point system arising when the stress symmetry is weakly-imposed.
The coupling of solid and fluid phases is discretized by means of stable mixed elements in space and implicit time advancing schemes. 
The presented stability analysis is fully robust with respect to meaningful cases of degenerate model parameters.
Numerical tests validate the convergence and robustness and assess the performances of the method for the simulation of wave propagation phenomena in porous materials.
\end{abstract}

\section{Introduction}\label{sec:intro}

The numerical modeling of coupled diffusion-deformation mechanisms in poroelastic media is relevant in several applications in geosciences, including subsidence due to fluid withdrawal, earthquake triggering due to pressure-induced faults slip, injection-production cycles in geothermal fields, and carbon dioxide storage in saline aquifers.
%The earliest theory modeling the effects of the pore fluid on the deformation of the soil was developed in the pioneering work of Terzaghi \cite{Terzaghi:43}, who proposed a model for consolidation accounting for the fluid-to-solid interaction only. In this case, the problem can be decoupled and solved in two stages. This kind of theory can successfully model some of the poroelastic processes in the case of highly compressible fluids such as air. However, when dealing with slightly compressible (or incompressible) fluids, the solid-to-fluid interaction cannot be neglected, since the changes in the stress can significantly influence the pore pressure. The first detailed mathematical theory of poroelasticity incorporating two way interactions was formulated by Biot \cite{Biot:41}. 
The poroelastic model has been consistently studied in its quasi-static formulations, where the acceleration effects are negligible. Some example of robust discretization schemes for the Biot's consolidation problems in its quasi-static form can be found in \cite{Boffi.Botti.ea:16, Chen2013,  Cockburn2009, Hu2017, Oyarzua2016}. However, for some practical applications, the quasi-static model is not sufficient for giving a comprehensive understanding of the phenomena and acceleration terms in both the solid and fluid phases need to be considered. This is the case of earthquake studies and the the studies of induced seismicity phenomena arising from soil exploitation activities. This model has been introduced by Biot in \cite{biot1} and it has been show in \cite{Carcione2010} that the dynamic effects are crucial for understanding the wave propagation phenomena in poroelastic media. 
%The model proposed by Biot was subsequently re-derived via homogenization \cite{Burridge.Keller:81} and mixture theory \cite{Bowen:82}, which placed it on a rigorous basis.

In the literature the Biot's poroelasticity model has been studied in different formulations, e.g. with the acceleration term introduced for the displacement field only \cite{Bonetti2025}, or coupled with different physics, e.g. the fully-dynamic thermo-poroelastic model \cite{Bonetti2023, Carcione2019}.
Different method and approaches have been developed for studying this problem. First, these methods include finite difference, finite element,
boundary element, finite volume, and spectral methods \cite{Chen:95, Gaspar2003, Lemoine:14, Morency2008, Santos1986}. More recent approaches
have considered high-order space-time continuous and discontinuous Galerkin methods \cite{Antonietti2022, Botti.Mazzieri:24} (also on polygonal grids), ADER scheme \cite{DeLaPuente2008}, mixed finite element methods \cite{Lee:23}, and hybridizable discontinuous Galerkin (HDG) method \cite{Meddahi:25}.

In this work, we introduce and analyze a fully-mixed formulation of the dynamic poroelasticity problem having the stress tensor, the fluid flux, the displacement, and the pore pressure as unknowns. The symmetry constraint on the stress field is imposed by adding to the four-field formulation a penalization term depending on the skew-symmetric part of the stress and a penalization parameter that has to be properly selected. The main novel contributions include a robust stability analysis at both continuous and discrete level which encompasses all the meaningful cases of degenerate model coefficients. 
The presented stability estimates does not depend on the dilatation and storage coefficients, showing the robustness with respect to volumetric locking and quasi-incompressible porous materials. Moreover, the a priori analysis also supports the case of vanishing densities parameters which make the model degenerate towards the quasi-static case. This is in contrast with respect to \cite{Lee:23} and \cite{Meddahi:25} were both the storage and densities coefficients are assumed to be strictly positive. Finally, we also mention that the results are obtained under weak data regularity and general boundary conditions. The case of mixed and non-homogeneous boundary conditions is often left out in the analysis of numerical schemes for poroelasticity problems, but is relevant for the practical applications and place several additional difficulties in the a priori analysis.  

The rest of the paper is organised as follows.
In Section \ref{sec:poro}, we present the dynamic poroelasticity problem written as a first-order hyperbolic system.
In Section \ref{sec:weakform}, we formulate the problem in variational form and present the main stability result.
Section \ref{sec:discrete.setting} describes the discrete setting and contains the statement of the discrete problem.
Finally, Section \ref{sec:numerical_results} contains a thorough numerical validation of the convergence and robustness properties of the method. A physically-consistent test on wave propagation in poroelastic media is also considered.

\section{Dynamic poroelasticity}\label{sec:poro}
In this section we present the low-frequency poroelasticity equations. We define the model data and parameters characterizing the initial-boundary value problem and then reformulate it as a first-order fully-coupled hyperbolic system.

\subsection{Model problem}\label{sec:model_pb}

We consider wave propagation through fluid-saturated porous media in the low-frequency regime \cite{biot1}. Letting $\phi\in (0,1)$ denote the reference material porosity, we assume a continuous superposition of solid and fluid phases occupying $(1-\phi)$ and $\phi$ of the total volume $|\Omega|$, respectively. We refer the reader to \cite{carcione2014, Coussy:04} for the precise derivation of the dynamic problem presented hereafter.

Let $\Omega\subset\Real^d$, $d\in\{2,3\}$, denote a bounded connected polyhedral domain with boundary $\partial\Omega$ and outward normal $\normal$. For a given time interval $(0,\tF]$, with $\tF>0$, 
the problem data consist of the forcing term  $\widetilde{\btens{\eta}}:\Omega\times(0,\tF)\to\Real^{d\times d}$,
volumetric loads $\bvec{f}, \, \bvec{h}:\Omega\times(0,\tF)\to\Real^d$, 
fluid source $g:\Omega\times(0,\tF)\to\Real$,
initial solid displacement $\bvec{d}_0:\Omega\to\Real^d$,
initial solid and filtration velocities $\bvec{u}_0, \bvec{w}_0:\Omega\to\Real^d$,
and initial pore pressure $p_0:\Omega\to\Real$.
Given the previous data, the poroelasticity initial-boundary value problem consists in finding a 
displacement field $\bvec{d}:\Omega\times\lbrack 0,\tF)\to\Real^d$,
velocity field $\bvec{w}:\Omega\times\lbrack 0,\tF)\to\Real^d$, and a 
pore pressure field $p:\Omega\times\lbrack 0,\tF)\to\Real$ such that
\begin{subequations}\label{eq:biot:strong}
  \begin{alignat}{2}
    \label{eq:biot:strong:momentum}
    \rho_u \partial_{tt} \bvec{d} + \rho_f \partial_{t}\bvec{w} -\DIV{\btens{\sigma}} &= \bvec{f} &\qquad&\text{in $\Omega\times (0,\tF)$},
    \\
    \label{eq:biot:strong:stress}
    \btens{\sigma} - \mathcal{C}\GRADs\bvec{d} +\alpha p \btens{I}_d   &= \widetilde{\btens{\eta}} &\qquad&\text{in $\Omega\times (0,\tF)$},
     \\
    \label{eq:biot:strong:darcydyn}
    \rho_f \partial_{tt} \bvec{d} + \rho_w \partial_{t}\bvec{w} + \diff^{-1}\bvec{w} +\GRAD p &= \bvec{h} &\qquad&\text{in $\Omega\times (0,\tF)$},
    \\
    \label{eq:biot:strong:masscons}
    s_0 \partial_t p + \alpha \partial_t (\DIV\bvec{d}) + \DIV\bvec{w} &= g &\qquad&\text{in $\Omega\times (0,\tF)$},
    \\
    \label{eq:biot:strong:initial_up}
    \bvec{d}(\cdot,0)=\bvec{d}_0, \quad p(\cdot,0) &=p_0 &\qquad&\text{in $\Omega$},
     \\
    \label{eq:biot:strong:initial_vw}
    \partial_t\bvec{d}(\cdot,0)=\bvec{u}_0, \quad \bvec{w}(\cdot,0) &=\bvec{w}_0 &\qquad&\text{in $\Omega$}. 
  \end{alignat}
\end{subequations}
  
  In the momentum equilibrium equation \eqref{eq:biot:strong:momentum}, $\partial_t$ and $\partial_{tt}$ denotes the first and second time derivative, while the parameter $\rho_u$ 
  is related to the solid and fluid density $\rho_s, \rho_f>0$ according to $\rho_u \coloneq \phi\rho_f + (1-\phi)\rho_s>0$.
  
  In the constitutive relation for the stress \eqref{eq:biot:strong:stress}, $\GRADs$ denotes the symmetric part of the gradient operator acting on vector-valued fields, 
  $\mathcal{C}:\Omega\to\Real^{d^4}$ is the uniformly elliptic fourth-order tensor expressing the linear stress-strain law, and $\alpha\in(\phi,1]$ is the Biot--Willis coefficient. 
  For isotropic materials, the tensor $\mathcal{C}$ can be expressed in terms of the Lam\'e parameters $\mu\in[\underline{\mu},\overline{\mu}]$, with $0<\underline{\mu}<\overline{\mu}$, 
  and $\lambda\ge 0$ as
  $$
    \mathcal{C}\btens{\tau} \coloneq 2\mu\ \btens{\tau} + \lambda\ {\rm tr}(\btens{\tau})\btens{I}_d
    \quad\text{ for all $\tau\in \Real^{d\times d}$},
  $$
  where $\btens{I}_d$ denotes the identity matrix of $\Real^{d\times d}$ and 
  ${\rm tr}(\btens{\tau}) \coloneq \sum_i^d \tau_{ii}$. Introducing the bulk modulus $\kappa\coloneq2 d^{-1}\mu + \lambda$ 
  and the deviatoric operator $\btens{\rm dev}(\btens\tau)\coloneq\tau - d^{-1}{\rm tr}(\btens{\tau})\btens{I}_d$,
  an alternative expression for $\mathcal{C}$ is given by 
  \begin{equation}\label{eq:cauchytensor}
    \mathcal{C}\btens{\tau} = 2\mu\ \btens{\rm dev}(\btens\tau) + \kappa\ {\rm tr}(\btens{\tau})\btens{I}_d
    \quad\text{ for all $\tau\in \Real^{d\times d}$}.
  \end{equation}
  
  In the dynamic Darcy equation~\eqref{eq:biot:strong:darcydyn}, the parameter $\rho_w$ is such that $\rho_w\coloneq \rho_f\nu\phi^{-1}$, where $\nu>1$ denotes the pores tortuosity, 
  while $\diff:\Omega\to\Real^{d\times d}_\symm$ is the hydraulic conductivity tensor which, for strictly positive real numbers $\underline{K}\le\overline{K}$, satisfies 
  \begin{equation}\label{eq:ass_diff}
    \text{ $\underline{K}\seminorm{\bvec{\xi}}^2\le\diff(\bvec{x})\bvec{\xi}\cdot\bvec{\xi}\le
    \overline{K}\seminorm{\bvec{\xi}}^2$ for every  $\bvec{x}\in\Omega$ and all $\bvec{\xi}\in\Real^d$.
    }
  \end{equation}
  Clearly, as a result of \eqref{eq:ass_diff}, the uniquely defined inverse tensor $\diff^{-1}$ is also uniformly elliptic in the sense that
  $$
  \text{ ${\overline{K}\,}^{-1} \seminorm{\bvec{\xi}}^2 \le\diff^{-1}(\bvec{x})\bvec{\xi}\cdot\bvec{\xi}\le
   \underline{K}^{-1} \seminorm{\bvec{\xi}}^2$ for every  $\bvec{x}\in\Omega$ and all $\bvec{\xi}\in\Real^d$.
    }
  $$
  To be valid, the equation~\eqref{eq:biot:strong:darcydyn} requires that the spectrum of the waves involve frequencies lower than a threshold depending on $\diff$, $\rho_f$, and $\phi$ (cf. \cite{Chiavassa:13}).  
  Otherwise, more complex models such as the one described in \cite{Lu.Hanyga:05} are required.
  
  Finally, in the mass conservation equation~\eqref{eq:biot:strong:masscons},  $s_0\ge 0$ denotes the 
  constrained storage coefficient measuring the amount of fluid that can be forced into the medium 
  by pressure increments due to the structure compressibility. The relevant case of incompressible grains 
  corresponds to the limit value $s_0=0$.
  
We close the problem by prescribing the boundary conditions. We introduce two partitions $\partial \Omega = \Gamma_d \cup \Gamma_\sigma = \Gamma_p\cup \Gamma_w$ such that $\Gamma_d \cap \Gamma_\sigma=\Gamma_p\cap \Gamma_w=\emptyset$, assuming that the measures of $\Gamma_d$ and $\Gamma_p$ are strictly positive. We impose a given traction on $\Gamma_\sigma$, a fixed displacement on $\Gamma_d$, a given flux on $\Gamma_w$, and a pressure on $\Gamma_p$, namely
\begin{subequations}\label{eq:boundaryconds}
    \begin{alignat}{2}
 \label{eq:neumann}
&\btens{\sigma}\ \bvec{n}= \bvec{t}  \quad \text{on }\Gamma_\sigma\times (0,\tF] \qquad &\bvec{w}\cdot\bvec{n} = w_n \quad\text{on }\Gamma_w\times (0,\tF] 
\\
\label{eq:dirichlet}
&\bvec{d} = \bvec{d}_d \quad\text{on }\Gamma_d\times (0,\tF] \qquad &p= p_d \quad\text{on }\Gamma_p\times (0,\tF].
\end{alignat}
\end{subequations}

\subsection{First-order hyperbolic system}

Different formulations of problem \eqref{eq:biot:strong} have been studied in the literature, including the frequency-domain displacement-pressure formulation of \cite{Chen:95};
the two-displacement formulation \cite{matuszy:14,Botti.Mazzieri:24}, that is obtained by inserting the expression of the total stress $\btens{\sigma}$ and the pore pressure $p$ from \eqref{eq:biot:strong:stress} and \eqref{eq:biot:strong:masscons} in the other equations; and the three-field formulations of \cite{Anselman:23} considering the solid displacement, solid velocity, and pressure as unknowns.
In what follows, we focus on the first-order fully-mixed formulation considered, e.g., in \cite{Lee:23, Lemoine:14}. 
This choice leads to several advantages: \textit{(i)} it produces a better approximation of the Darcy velocity and stress fields; \textit{(ii)} mass and momentum conservation are locally satisfied; \textit{(iii)} additional assumptions on the poroelastic coeffcient are not required; and \textit{(iv)} it is easier to analyze the problem in the framework of strongly continuous semigroups (cf. \cite{Meddahi:25}).

We introduce the solid velocity $\bvec{u}=\partial_t \bvec{d}$ and assume that the initial displacement field $\bvec{d}_0$ is sufficiently regular to associate the initial stress tensor, defined according to \eqref{eq:biot:strong:stress} as $\btens{\sigma}_0 = \mathcal{C}\GRADs\bvec{d}_0 -\alpha p_0 \btens{I}_d$.
The uniformly elliptic fourth-order compliance tensor $\mathcal{A}$ is defined such that $\mathcal{A}(\mathcal{C}\btens{\tau})=\mathcal{C}(\mathcal{A}\btens{\tau})=\btens{\tau}$ for all $\btens{\tau}\in\Real^{d\times d}$. The explicit expression of the action of $\mathcal{A}$ derived from \eqref{eq:cauchytensor} is given by
\begin{equation}\label{eq:compliance}
\mathcal{A}\btens\tau 
  \coloneq \frac{\btens{\rm dev}(\btens\tau)}{2\mu} + \frac{{\rm tr}(\btens{\tau})}{d^2 \kappa} \btens{I}_d.
\end{equation}
We also observe that, taking the trace of the tensor equation \eqref{eq:biot:strong:stress}, allows to express the divergence of the solid displacement as 
$$
\kappa \DIV\bvec{d} = \alpha p+  d^{-1}{\rm tr}(\btens{\sigma}) .
$$
Thus, taking the time derivative in the previous identity and in \eqref{eq:biot:strong:stress}, plugging the resulting expression into \eqref{eq:biot:strong:masscons}, and owing to the definition of the compliance tensor \eqref{eq:compliance}, we rewrite the dynamic poroelasticity problem as a first-order system:
find the velocities $\bvec{u}, \bvec{w}$, the stress tensor $\btens{\sigma}$, and pressure $p$ such that
\begin{subequations}\label{eq:poro:strong}
  \begin{alignat}{2}
    \label{eq:poro:strong:momentum}
    \rho_u \partial_{t} \bvec{u} + \rho_f \partial_{t}\bvec{w} -\DIV{\btens{\sigma}} &= \bvec{f} &\qquad&\text{in $\Omega\times (0,\tF)$},
    \\
    \label{eq:poro:strong:darcydyn}
    \rho_f \partial_{t} \bvec{u} + \rho_w \partial_{t}\bvec{w} + \diff^{-1}\bvec{w} +\GRAD p &= \bvec{h} &\qquad&\text{in $\Omega\times (0,\tF)$},
    \\
     \label{eq:poro:strong:stress}
    \mathcal{A}(\partial_t\btens{\sigma}) +\alpha (d\kappa)^{-1} \partial_t p \btens{I}_d  - \GRADs\bvec{u}  &= \btens{\eta} &\qquad&\text{in $\Omega\times (0,\tF)$},
    \\
    \label{eq:poro:strong:masscons}
    (s_0+ \alpha^2\kappa^{-1}) \partial_t p + \alpha(d\kappa)^{-1}  {\rm tr}(\partial_t\btens{\sigma}) + \DIV\bvec{w} &= g &\qquad&\text{in $\Omega\times (0,\tF)$},
    \\
    \label{eq:poro:strong:initial}
    \bvec{u}(\cdot,0)=\bvec{u}_0, \;\  \bvec{w}(\cdot,0) =\bvec{w}_0\;\ \btens{\sigma}(\cdot,0)=\btens{\sigma}_0, \;\ p(\cdot,0) &=p_0 &\qquad&\text{in $\Omega$},
  \end{alignat}
\end{subequations}
where the right-hand side term in \eqref{eq:poro:strong:stress} is defined by $\btens{\eta}=\partial_t\widetilde{\btens{\eta}}$.

The well-posedness of problem \eqref{eq:poro:strong} has been established in \cite[Section 3]{Meddahi:25} by applying the Hille--Yosida theory. One of the main ingredient needed to establish a stability result for the previous system consists in pointing out the the coupling between the solid and filtration velocities as well as the coupling between the pressure and the stress trace are symmetric and positive, in the sense that the interaction matrices
\begin{equation}\label{eq:R12}
R_1 \coloneq \begin{pmatrix}
\rho_u & \rho_f\\
\rho_f & \rho_w
\end{pmatrix}
\qquad\text{and }\quad
R_2 \coloneq \frac{1}{d\kappa} \begin{pmatrix}
d^{-1} \;& \alpha \\
\alpha \; & s_0 d\kappa + d\alpha^2 
\end{pmatrix}
\end{equation}
are elliptic. Indeed, due to the definition of the model parameters given in Section \ref{sec:model_pb}, we have ${\rm det}(R_1) = (\phi^{-1}-1)\rho_s \rho_f \nu + (\nu -1)\rho_f^2\ge0$ and ${\rm det}(R_2)=s_0 d^{-1}\ge0$.

\section{Variational formulation}\label{sec:weakform}
In this section we introduce some notations that will be used in the sequel, and derive the mixed variational formulation. First, we present a five-field formulation where the stress symmetry is enforce weakly as in \cite{Lee2016,Lee:23}. Then, we introduce a stabilization term that allows to rewrite the previous problem as a four-field system with a penalization of the skew-symmetric part of the stress.

\subsection{Notation for functional spaces}

Let $X\subset\overline{\Omega}$.
Spaces of functions, vector fields, and tensor fields defined over $X$ are respectively denoted by italic capital, boldface Roman capital, and special Roman capital letters.
The subscripts ``s'' and ``sk'' appended to a special Roman capital letter denotes a space of symmetric and skew-symmetric tensor fields, respectively.
Thus, for example, $L^2(X), \Lvec{X}$, and $\Lmatsk{X}$ respectively denote the spaces of square integrable functions, vector fields, and skew-symmetric tensor fields over $X$. 

For any measured set $X$ and any $m\in\Integer$, we denote by $H^m(X)$ the usual Sobolev space of functions that have weak partial derivatives of order up to $m$ in $L^2(X)$. 
We denote by $(\cdot,\cdot)_X$ and $(\cdot,\cdot)_{m,X}$ the usual scalar products in $L^2(X)$ and $H^{m}(X)$ respectively, and by $\norm[X]{{\cdot}}$ and $\norm[m,X]{{\cdot}}$ the induced norms. 
The subscript $X$ is omitted whenever $X=\Omega$.
The notation $\langle \varphi,\xi \rangle$ is used for the duality product between two functions $\varphi\in H^{\frac12}(\partial\Omega)$ and $\xi\in H^{-\frac12}(\partial\Omega)$.
Additionally, let $\Hdiv{X}$ be the subspace of $\Lvec{X}$ spanned by vector-valued functions admitting a weak divergence in $L^2(X)$ with natural norm
$$
\norm[{\rm div},X]{\bvec{v}} \coloneq \left(\norm[X]{\bvec{v}}^2 + h_X^2\norm[X]{\DIV\bvec{v}}^2 \right)^{\frac12},
$$
with $h_X$ denoting the diameter of $X$.
We define $\Hdivmat{X}$ to be the space of tensor-valued functions in $\Lmat{X}$ with weak divergence in $\Lvec{X}$. 
%According to the previous notations, we also introduce $\Hdivmats{X}\coloneq \Hdivmat{X}\cap \Lmats{X}$.

For a vector space $V$ with scalar product $(\cdot,\cdot)_V$, the space $C^{m}(V)\coloneq C^{m}([0,\tF];V)$ is spanned by $V$-valued functions that are $m$-times continuously differentiable in the time interval $[0,\tF]$. The space $C^{m}(V)$ is a Banach space equipped with the norm 
$$
\norm[C^{m}(V)]{\varphi}\coloneq\max_{0\le i\le m}\max_{t\in[0,\tF]}\norm[V]{\partial_t^{i}\varphi(t)}.
$$ 
Similarly, the Hilbert space $H^{m}(V)\coloneq H^{m}((0,\tF);V)$ is spanned by $V$-valued functions of the time interval,  and the norm $\norm[H^{m}(V)]{{\cdot}}$ is induced by the scalar product
$$
(\varphi,\psi)_{H^{m}(V)}=\sum_{j=0}^m \int_0^{\tF} (\partial_t^{j}\varphi(t),\partial_t^{j}\psi(t))_V \ud t
\qquad \text{ for all }\varphi,\psi \in H^{m}(V).
$$

In order to match the boundary conditions in \eqref{eq:boundaryconds}, we introduce $\HvecD{\Omega}$ and $H^1_{0,\Gamma_p}(\Omega)$ as the subspaces of $\Hvec{\Omega}$ and $H^1(\Omega)$ spanned by functions having zero trace on $\Gamma_d$ and $\Gamma_p$, respectively; and define the Hilbert spaces
\begin{align*}
\bvec{\Sigma} &\coloneq\{\btens{\tau} \in \Hdivmat{\Omega}\,:\,\langle\btens{\tau}\ \bvec{n},\bvec{v}\rangle_{\partial\Omega}=0 \ \ \forall \bvec{v}\in \HvecD{\Omega}\},\; 
&& \bvec{U} :=\Lvec{\Omega},
\\
\bvec{W} &:= \{\bvec{z} \in \Hdiv{\Omega}\,:\,\langle\bvec{w}\cdot \bvec{n}, \xi\rangle_{\partial\Omega}=0 \ \ \forall \xi\in H^1_{0,\Gamma_p}(\Omega)\},\;
&& P:= L^2(\Omega).
\end{align*}
For later use, we also define the skew operator: $ \btens{\rm skw}(\btens\tau)=\frac{\btens{\tau}-{\btens\tau}^{\rm T}}2$ for all $\btens{\tau}\in\Real^{d\times d}$, and introduce the short-handed notation $\dot{\varphi} = \partial_t{\varphi}$ for the time derivative. Due to the definition of the ${\rm tr}$, $\btens{\rm skw}$, and $\btens{\rm dev}$ operators, we observe that for all $\btens{\tau},\btens\xi\in\Real^{d\times d}$ 
$$
\begin{aligned}
\btens\tau:{\rm tr}(\btens\xi)\Id &= {\rm tr}(\btens\xi){\rm tr}(\btens\tau), \\
\btens\tau:\btens{\rm skw}(\btens\xi) &= 
\btens{\rm skw}(\btens\tau):\btens{\rm skw}(\btens\xi), \\
\btens\tau:\btens{\rm dev}(\btens\xi) &= 
\btens{\rm dev}(\btens\tau):\btens{\rm dev}(\btens\xi),
\end{aligned}
$$
with the symbol ``$\,:\,$'' denoting the Frobenius product.
In the derivation of the weak formulation of \eqref{eq:poro:strong} in the next section, we will use the previous identities multiple times in both directions. 
For the sake of brevity, throughout the paper we will use the notation $a\lesssim b$ for the inequality $a\le Cb$ with a generic constant $C>0$ independent of the discretization parameters and the model coefficients $\lambda, \rho_s,\rho_f,\rho_w$ and $s_0$.

\subsection{Weak-symmetry and skew-symmetry penalization}

In the mixed formulation of elasticity and poroelasticity problems, the symmetry of the stress tensor $\btens{\sigma}$ has to be explicitly enforced, either strongly in the functional space (see, e.g., \cite{Arnold:02, Yi:14, Visinoni:25}), or weakly by the use of a Lagrange multiplier as in \cite{Arnold.Falk:07,Lee2016, Lee:23}.
In what follows, we opt for this second strategy since it leads to a formulation that can be discretized with low order standard finite element spaces.

We introduce the skew-symmetric part of $\GRAD\bvec{u}$, i.e. $\btens{\rm skw}(\GRAD\bvec{u})$, as additional unknown, named $\btens{\zeta}$, we add to the hyperbolic system \eqref{eq:poro:strong} the equation $\btens{\rm skw}(\dot{\btens{\sigma}})=\bvec{0}$, and we replace equation \eqref{eq:poro:strong:stress} by
$$
\mathcal{A}\dot{\btens{\sigma}} +\alpha (d\kappa)^{-1} \dot{p} \btens{I}_d  - \GRAD\bvec{u}
+ \btens{\zeta} = \btens{\eta} \qquad\text{in $\Omega\times (0,\tF)$}.
$$

The weak formulation of the dynamic poroelasticity problem \eqref{eq:poro:strong} is derived by testing with suitable functions and integrating by parts in \eqref{eq:poro:strong:darcydyn} and \eqref{eq:poro:strong:momentum}. 
Let $\btens{\eta}\in L^2(\Lmat{\Omega})$, $\bvec{f}\in L^2(\Lvec{\Omega})$, $\bvec{h}\in L^2(\Lvec{\Omega})$, $g\in L^2(L^2(\Omega))$, $\bvec{u}_d=\dot{\bvec{d}}_d\in L^2(\bvec{H}^{\frac12}(\partial\Omega))$, and $p_d\in L^2(H^{\frac12}(\partial\Omega))$. For simplicity, we assume
that $\bvec{t} = \bvec{0}$ and $w_n=0$ in \eqref{eq:neumann}. The general case of non-homogeneous essential conditions can be obtained by the $\Hdivmat{\Omega}$-regular lifting of boundary data as in \cite[Appendix A]{Botti.Visinoni:25}. 
The weak formulation with weakly enforced stress symmetry reads: 
for almost every $t\in(0,\tF]$, find $(\bvec{u}(t),\bvec{w}(t), \btens{\sigma}(t), p(t), \btens{\zeta}(t)) \in\bvec{U}\times\bvec{W}\times\bvec{\Sigma}\times P \times \Lmatsk{\Omega}$ such that
\begin{subequations}
\label{eq:biot:weak1}
\begin{alignat}{2}
  \label{eq:biot:weak1:displacement}
  (\rho_u\dot{\bvec{u}} + \rho_f\dot{\bvec{w}}, \bvec{v}) - (\DIV\btens{\sigma},\bvec{v})  
  &=(\bvec{f},\bvec{v}) &\quad&\forall \bvec{v}\in \bvec{U},
  \\
  \label{eq:biot:weak1:flux}
  (\rho_f\dot{\bvec{u}} + \rho_w\dot{\bvec{w}},\bvec{z})\hspace{-.5mm}
  +\hspace{-.5mm}(\diff^{-1}\bvec{w},\bvec{z}) -(p,\DIV\bvec{z}) &= 
  (\bvec{h},\bvec{z}) \hspace{-.5mm} + \hspace{-.5mm}
  \langle p_d,\bvec{z}\cdot\bvec{n}\rangle
  &\quad&\forall \bvec{z}\in\bvec{W}, 
  \\
  \label{eq:biot:weak1:stress}
  (\mathcal{A}\dot{\btens{\sigma}},\btens\tau)\hspace{-.5mm} + \hspace{-.5mm}
  \Big(\alpha \dot{p}, \frac{{\rm tr}(\btens\tau)}{d\kappa}\Big)  
  \hspace{-.5mm} + \hspace{-.5mm} (\bvec{u},\DIV\btens\tau) 
  \hspace{-.5mm} + \hspace{-.5mm} (\btens{\zeta},\btens\tau)&= 
  (\btens{\eta},\btens\tau)
  +\langle \bvec{u}_d,\btens{\tau}\bvec{n}\rangle 
  &\quad&\forall \btens\tau\in \bvec{\Sigma},
  \\
  \label{eq:biot:weak1:pressure}
  \Big(\frac{s_0\kappa+\alpha^2}{\kappa}\dot{p},q\Big) \hspace{-.5mm}
  + \hspace{-.5mm}
  \Big(\frac{{\rm tr}(\dot{\btens{\sigma}})}{d\kappa}, \alpha q\Big)\hspace{-.5mm}
  + \hspace{-.5mm}(\DIV\bvec{w},q) &= (g,q) &\quad&\forall q\in P,
  \\
  \label{eq:biot:weak1:rotation}
  -(\btens{\rm skw}(\dot{\btens{\sigma}}),\btens\xi) &= 0 &\forall \btens\xi &\in \Lmatsk{\Omega}.
\end{alignat}
\end{subequations}

Recalling the definition of $\mathcal{A}$ in \eqref{eq:compliance}, observing that $\btens{\zeta}$ is skew-symmetric, and defining the linear functional $\mathcal{L}_{\btens{\Sigma}}:\bvec{\Sigma}\to\Real$ defined as 
$\mathcal{L}_{\btens{\Sigma}}(\btens{\tau})= (\btens{\eta},\btens\tau)
  +\langle \bvec{u}_d,\btens{\tau}\bvec{n}\rangle$,
we rewrite equation \eqref{eq:biot:weak1:stress} as
\begin{equation}\label{eq:manipule1}
\Big(\frac{\btens{\rm dev}(\dot{\btens\sigma})}{2\mu},\btens\tau\Big)+
\Big(\frac{{\rm tr}(\dot{\btens\sigma}) + \alpha d \dot{p}}{d^2\kappa}, {\rm tr}(\btens\tau)\Big)+
(\btens{\zeta}, \btens{\rm skw}(\btens\tau)) + (\bvec{u},\DIV\btens\tau) =
\mathcal{L}_{\btens{\Sigma}}(\btens{\tau}).
\end{equation}

Now, we introduce a perturbation to the weak-symmetry constraint \eqref{eq:biot:weak1:rotation} by letting $\varepsilon:\Omega\to\Real^+$
\textit{small enough}, and replacing \eqref{eq:biot:weak1:rotation} with
\begin{equation}\label{eq:stab}
(\varepsilon\ \btens{\zeta},\btens\xi) - (\btens{\rm skw}(\dot{\btens{\sigma}}),\btens\xi) = 0.
\end{equation}
From the previous relation, we deduce $(\btens{\zeta},\btens\xi)= (\varepsilon^{-1}\btens{\rm skw}(\dot{\btens{\sigma}}),\btens\xi)$, that, owing to $\btens{\rm skw}(\btens\tau)\in \Lmatsk{\Omega}$ for all $\btens\tau\in \bvec{\Sigma}$, we can insert in \eqref{eq:manipule1} to obtain
\begin{equation*}%\label{eq:manipule2}
\Big(\frac{\btens{\rm dev}(\dot{\btens\sigma})}{2\mu},\btens{\rm dev}(\btens\tau)\Big)+
\Big(\frac{{\rm tr}(\dot{\btens\sigma}) + \alpha d \dot{p}}{d^2\kappa}, {\rm tr}(\btens\tau)\Big)+
\Big(\frac{\btens{\rm skw}(\dot{\btens{\sigma}})}{\varepsilon}, \btens{\rm skw}(\btens\tau)\Big) 
+ (\bvec{u},\DIV\btens\tau) \hspace{-.5mm}= \hspace{-.5mm}
\mathcal{L}_{\btens{\Sigma}}(\btens{\tau}).
\end{equation*}

Denoting the right-hand side term appearing in \eqref{eq:biot:weak1:flux} as the linear functional $\mathcal{L}_{\bvec{W}}(\bvec{z})= (\bvec{h},\bvec{z}) + \langle p_d,\bvec{z}\cdot\bvec{n}\rangle$ and recalling the definitions of the matrices $R_1$ and $R_2$ in \eqref{eq:R12}, we can transform the weak problem \eqref{eq:biot:weak1} into a four-field formulation reading:
for almost every $t\in(0,\tF]$, find $(\bvec{u}(t),\bvec{w}(t), \btens{\sigma}(t), p(t))\in\bvec{U}\times\bvec{W}\times\bvec{\Sigma}\times P$ such that
\begin{subequations}
\label{eq:weak2}
\begin{align}
  \label{eq:weak2:DispFlux}
  \left(R_1\begin{bmatrix}
     \dot{\bvec{u}}    \\
     \dot{\bvec{w}}   
  \end{bmatrix}, 
  \begin{bmatrix}
     \bvec{v}    \\
     \bvec{z}   
  \end{bmatrix}\right)
  +(\diff^{-1}\bvec{w},\bvec{z})
  - \left(\begin{bmatrix}
     \DIV\btens{\sigma}   \\
     p  
  \end{bmatrix},
  \begin{bmatrix}
     \bvec{v}   \\
     \DIV\bvec{z}  
  \end{bmatrix}\right)  \hspace{-.5mm}
  &=\hspace{-.5mm}\begin{bmatrix}
     (\bvec{f},\bvec{v})  \\
     \mathcal{L}_{\bvec{W}}(\bvec{z})
  \end{bmatrix}
  \\
  \label{eq:weak2:StressPress}
  \Big(\frac{\btens{\rm dev}(\dot{\btens\sigma})}{2\mu}\hspace{-.5mm}+\hspace{-.5mm}
  \frac{\btens{\rm skw}(\dot{\btens{\sigma}})}{\varepsilon}
  ,\btens\tau\Big) \hspace{-.5mm}+\hspace{-.5mm}
\left(\hspace{-.2mm} R_2\hspace{-.6mm}\begin{bmatrix}
     {\rm tr}(\dot{\btens\sigma})\hspace{-.5mm}  \\
     \dot{p}
  \end{bmatrix} \hspace{-.5mm},\hspace{-.5mm} 
  \begin{bmatrix}
     {\rm tr}(\btens\tau) \hspace{-.5mm}  \\
     q  
  \end{bmatrix}\right)
\hspace{-.5mm}+\hspace{-.5mm} \left(\begin{bmatrix}
     \bvec{u}   \\
     \DIV\bvec{w} 
  \end{bmatrix} \hspace{-.5mm},\hspace{-.5mm}
  \begin{bmatrix}
      \DIV\btens{\tau}   \\
     q 
  \end{bmatrix}\right)\hspace{-.5mm}
  &=\hspace{-.5mm}\begin{bmatrix}
     \mathcal{L}_{\btens{\Sigma}}(\btens{\tau})  \\
     (g,q)
  \end{bmatrix}
\end{align}
\end{subequations}
for all $(\bvec{v},\bvec{z},\btens\tau,q)\in \bvec{U}\times\bvec{W}\times\bvec{\Sigma}\times P$.
Problem \eqref{eq:weak2} is completed with the initial conditions in \eqref{eq:poro:strong:initial} assuming that $\btens{\sigma}_0\in\Lmats{\Omega}$, $\bvec{u}_0, \bvec{w}_0\in\Lvec{\Omega}$, and $p_0\in L^2(\Omega)$. 

\begin{remark}
The idea of adding a small perturbation in equation \eqref{eq:stab} depending on the parameter $\varepsilon$ is resemblant to the artificial compressibility technique used to stabilized numerical schemes for incompressible fluid flow problem (see, e.g., \cite{Chorin:67}). 
\end{remark}

\subsection{Well-posedness}

This section investigates the well-posedness of the mixed variational problem
\eqref{eq:weak2}. Under additional regularity assumptions on the problem data (forcing terms, boundary data, and initial conditions) the existence of solutions can be inferred by applying the Hille--Yosida theorem as in \cite[Appendix A]{Antonietti.Botti.ea:21} and \cite[Section 3.2]{Meddahi:25}.
In what follows, we focus on proving a stability result under weaker assumptions, which in turn also implies the uniqueness of solutions.

First, we present some preliminary results.
We recall the following inf-sup conditions establishing the stability of the divergence operator.
\begin{lemma}[Inf-sup conditions]
There exist positive constants $\beta_{m}$ (only depending on $\Omega$ and $\Gamma_{\sigma}$) and $\beta_{f}$ (only depending on $\Omega$ and $\Gamma_{w}$) such that 
\begin{align}
\label{eq:infsup_mech}
\beta_{m} \norm[]{\bvec{v}}  &\le \sup_{\btens\tau\in\Hdivmats{\Omega}\cap\bvec{\Sigma}\setminus\{\bvec{0}\}} \frac{\left(\bvec{v},\DIV \btens\tau \right)}{\norm[{\rm div}]{\btens\tau}}
\qquad\forall \bvec{v}\in\bvec{U},
\\
\label{eq:infsup_flux}
\beta_{f} \norm[]{q}  &\le \sup_{\bvec{z}\in\bvec{W}\setminus\{\bvec{0}\}} 
\frac{(q,\DIV \bvec{z})}{\norm[{\rm div}]{\bvec{z}}}
\qquad\qquad\quad\;\ \forall q\in P.
\end{align}
\end{lemma}
\begin{proof}
The detailed proof together with the explicit dependence of the constants $\beta_m$ and $\beta_f$ with respect to the domain $\Omega$ in the case of fully-essential boundary conditions (i.e. $\Gamma_{\sigma}=\Gamma_{w}=\partial\Omega$) is given in \cite[Propositions 1.1 and 1.5]{Botti.Mascotto:25}.
For the alternative special case $\Gamma_{\sigma}=\emptyset$, we refer to \cite[Chapter 9]{Boffi.Brezzi.ea:13}. 
The general case can be obtain by reasoning as in \cite[Lemma A.1]{Beirao.Canuto:21} or \cite[Theorem 1.4]{Botti:25}.
\end{proof}
The next technical result is proved in \cite[Lemma 3]{Cancrini:25}.
\begin{lemma}[$\btens{{\rm dev}-{\rm div}}$ and trace inequalities]
\label{lemma:dev_div}
There exists two positive constants $C_{\rm dd}$ and $C_{\rm tm}$ (only depending on $\Omega$ and $\Gamma_{d}$), and a positive constant $C_{\rm tf}$ (only depending on $\Omega$ and $\Gamma_{p}$) such that
\begin{align}
\label{eq:devdiv}
\norm[]{\btens{\tau}}^2 &\le C_{\rm dd} \left( 
\norm[]{\btens{{\rm dev}}(\btens{\tau})}^2
+h_\Omega^2\norm[]{\DIV\btens{\tau}}^2\right)
\qquad\,\forall \btens{\tau}\in\btens{\Sigma},\\ 
 \label{eq:traceHdivmat}
\norm[\bvec{H}^{-1/2}(\partial\Omega)]{\btens{\tau}\bvec{n}}^2 
&\le C_{\rm tm} \left( 
\norm[]{\btens{{\rm dev}}(\btens{\tau})}^2
+h_\Omega^2\norm[]{\DIV\btens{\tau}}^2\right)
\qquad\, \forall \btens{\tau}\in\btens{\Sigma}, \\
 \label{eq:traceHdiv}
\norm[H^{-1/2}(\partial\Omega)]{\bvec{z}\cdot\bvec{n}} 
&\le C_{\rm tf} 
\norm[{\rm div}]{\bvec{z}}
\qquad\qquad\qquad\qquad\qquad\;\;\,\forall \bvec{z}\in\bvec{W}.
\end{align}
As a consequence, $\norm[\btens\Sigma]{\btens\tau}\coloneq
(\norm[]{\btens{{\rm dev}}(\btens{\tau})^2}
+h_\Omega^2\norm[]{\DIV\btens{\tau}}^2)^{\frac12}$ defines a norm on $\btens\Sigma$.
\end{lemma}

The previous Lemma is needed to bound the right-hand side terms in \eqref{eq:weak2}. In the next result, given an Hilbert space $\bvec{H}$, we denote by $\norm[\bvec{H}^*]{\cdot}$ the dual norm defined as
$$
\norm[\bvec{H}^*]{\mathcal{L}} \coloneq\sup_{\bvec\xi\in\bvec{H}\setminus\{\bvec{0}\}} 
\frac{\mathcal{L}(\bvec{\xi})}{\norm[\bvec{H}]{\bvec\xi}}.
$$
\begin{lemma}[Bounds for $\mathcal{L}_{\btens{\Sigma}}$ and $\mathcal{L}_{\bvec{W}}$]
    Assume that the forcing and boundary terms 
    $\btens{\eta}$, $\bvec{h}$, $\bvec{u}_d$, and $p_d$ are such that,
    for all $t\in(0,\tF]$, 
    $$
    \btens{\eta}(t)\in \Lmat{\Omega},\quad
    \bvec{h}(t)\in \Lvec{\Omega}, \quad 
    \bvec{u}_d(t)\in \bvec{H}^{1/2}(\partial\Omega), \quad 
    p_d(t)\in H^{1/2}(\partial\Omega).
    $$  
    Then, for all $t\in(0,\tF]$ the linear functionals $\mathcal{L}_{\btens{\Sigma}}(t)$ and $\mathcal{L}_{\bvec{W}}(t)$ in \eqref{eq:weak2} are such that
    \begin{equation}\label{eq:bound.rhs}
    \begin{aligned}
        \norm[\bvec{W}^*]{\mathcal{L}_{\bvec{W}}(t)}&\le \big(
        \norm[]{\bvec{h}(t)} + C_{\rm tf}\norm[H^{1/2}(\partial\Omega)]{p_d(t)}\big), 
        \\
        \norm[\btens\Sigma^*]{\mathcal{L}_{\btens{\Sigma}}(t)}&\le
        \big(
        C_{\rm dd}\norm[]{\btens{\eta}(t)} + C_{\rm tm}\norm[\bvec{H}^{1/2}(\partial\Omega)]{\bvec{u}_d(t)}\big).
    \end{aligned}
    \end{equation}
\end{lemma}
\begin{proof}
    The first bound in \eqref{eq:bound.rhs} directly follows from the application of the Cauchy-Schwarz inequality together with \eqref{eq:traceHdiv}. For the second, we recall the definition of $\norm[\btens\Sigma]{\btens\tau}$ and apply again the Cauchy-Schwarz inequality followed by \eqref{eq:devdiv} and \eqref{eq:traceHdivmat}.
\end{proof}

We aim to prove a robust stability bound, which also holds in the degenerate cases $s_0=0$, $\rho_f=0$, and $\lambda\to\infty$ (implying $\kappa\to\infty$). For this purpose, given a function $\varphi:[0,t_F]\to\Real$, we introduce the notation $\tilde{\varphi}$ for the integral function such that $\partial_t\tilde{\varphi}=\varphi$ and $\tilde{\varphi}(0)=0$, namely
$\tilde{\varphi}(t) = \int_0^{t_F} \varphi(s) \ud s$.
We adopt the same notation for time-dependent scalar, vector, and tensor fields.
Thus, integrating in time problem \eqref{eq:weak2}, leads to
\begin{subequations}
\label{eq:weak3}
\begin{align}
  \label{eq:weak3:DispFlux}
  \left(R_1\begin{bmatrix}
     \bvec{u}    \\
     \bvec{w}   
  \end{bmatrix}, 
  \begin{bmatrix}
     \bvec{v}    \\
     \bvec{z}   
  \end{bmatrix}\right)
  +(\diff^{-1}\tilde{\bvec{w}},\bvec{z})
  - \left(\begin{bmatrix}
     \DIV\tilde{\btens{\sigma}}   \\
     \tilde{p}  
  \end{bmatrix},
  \begin{bmatrix}
     \bvec{v}   \\
     \DIV\bvec{z}  
  \end{bmatrix}\right)  \hspace{-.5mm}
  &=\hspace{-.5mm}\begin{bmatrix}
     \mathcal{L}_1(\bvec{v})  \\
     \mathcal{L}_2(\bvec{z})
  \end{bmatrix}
  \\
  \label{eq:weak3:StressPress}
  \Big(\frac{\btens{\rm dev}(\btens\sigma)}{2\mu}\hspace{-.5mm}+\hspace{-.5mm}
  \frac{\btens{\rm skw}(\btens{\sigma})}{\varepsilon}
  ,\btens\tau\Big) \hspace{-.5mm}+\hspace{-.5mm}
\left(\hspace{-.2mm} R_2\hspace{-.6mm}\begin{bmatrix}
     {\rm tr}(\btens\sigma)\hspace{-.5mm}  \\
     p
  \end{bmatrix} \hspace{-.5mm},\hspace{-.5mm} 
  \begin{bmatrix}
     {\rm tr}(\btens\tau) \hspace{-.5mm}  \\
     q  
  \end{bmatrix}\right)
\hspace{-.5mm}+\hspace{-.5mm} \left(\begin{bmatrix}
     \bvec{d}   \\
     \DIV\tilde{\bvec{w}} 
  \end{bmatrix} \hspace{-.5mm},\hspace{-.5mm}
  \begin{bmatrix}
      \DIV\btens{\tau}   \\
     q 
  \end{bmatrix}\right)\hspace{-.5mm}
  &=\hspace{-.5mm}\begin{bmatrix}
     \mathcal{L}_3(\btens{\tau})  \\
     \mathcal{L}_4(q)
  \end{bmatrix}
\end{align}
\end{subequations}
for all $(\bvec{v},\bvec{z},\btens\tau,q)\in \bvec{U}\times\bvec{W}\times\bvec{\Sigma}\times P$, with right-hand side terms defined by
$$
\begin{aligned}
\mathcal{L}_1(\bvec{v}) &\coloneq (\tilde{\bvec{f}},\bvec{v}) 
+ (\rho_u\bvec{u}_0 + \rho_f\bvec{w}_0,\bvec{v})
\\
\mathcal{L}_2(\bvec{z}) &\coloneq \tilde{\mathcal{L}}_{\bvec{W}}(\bvec{z})
+ (\rho_f\bvec{u}_0 + \rho_w\bvec{w}_0,\bvec{z})
\\
\mathcal{L}_3(\btens{\tau}) &\coloneq \tilde{\mathcal{L}}_{\btens{\Sigma}}(\btens{\tau})
+ ((2\mu)^{-1}\btens{\rm dev}(\btens\sigma_0),\btens\tau) \hspace{-.5mm}+\hspace{-.5mm}
({\rm tr}(\btens\sigma_0) + \alpha d p_0, (d\kappa)^{-1} {\rm tr}(\btens\tau))
\hspace{-.5mm}+\hspace{-.5mm} (\bvec{d}_0,\DIV\btens\tau)  
\\
\mathcal{L}_4(q) &\coloneq (\tilde{g},q)
+ ((s_0+\kappa^{-1}\alpha^2) p_0,q)
+((d\kappa)^{-1}\alpha\ {\rm tr}(\btens{\sigma}_0), q).
\end{aligned}
$$
Clearly, the previously defined linear functionals $\mathcal{L}_{1-4}$ depend on the problem data $\bvec{f},\btens{\eta},\bvec{h},g,\bvec{u}_d,p_d$ and the initial conditions $\bvec{d_0},\bvec{u_0},\bvec{w_0},\btens{\sigma_0},p_0$. 
For the sake of conciseness, in the next theorem we will make the following assumption:

\medskip
\textit{Assumption 1.} The volumetric forcing and source terms as well as the boundary data appearing in \eqref{eq:weak2} and \eqref{eq:weak3} are $H^1$-regular in time, i.e.
$\btens{\eta}\in H^1(\Lmat{\Omega})$, $\bvec{f}\in H^1(\Lvec{\Omega})$, $\bvec{h}\in H^1(\Lvec{\Omega})$, $g\in H^1(L^2(\Omega))$, $\bvec{u}_d\in H^1(\bvec{H}^{\frac12}(\partial\Omega))$, and $p_d\in H^1(H^{\frac12}(\partial\Omega))$. Additionally, there is a positive constant $C_{\rm pb}$ depending on all the previous quantities, but independent of the possibly unbounded coefficients $s_0^{-1},\kappa,\rho_f^{-1}$, and $\rho_s^{-1}$, such that 
\begin{equation}\label{eq:databound}
\begin{aligned}
\norm[C^0(\Lvec{\Omega})]{\mathcal{L}_1}+
\norm[C^0(\bvec{W}^*)]{\mathcal{L}_2} + 
\norm[C^0(\btens{\Sigma}^*)]{\mathcal{L}_3}+
\norm[C^0(L^2(\Omega))]{\mathcal{L}_4}
&\le C_{\rm pb},\\
\norm[H^1(\Lvec{\Omega})]{\bvec{f}}+
\norm[H^1(\bvec{W}^*)]{\mathcal{L}_{\bvec{W}}} + 
\norm[H^1(\btens{\Sigma}^*)]{\mathcal{L}_{\btens{\Sigma}}}+
\norm[H^1(L^2(\Omega))]{g}
&\le C_{\rm pb}.
\end{aligned}
\end{equation}

In light of \eqref{eq:bound.rhs}, the $H^1$-regularity in time of the data, and the assumption on the initial conditions, we observe that \eqref{eq:databound} is not restrictive.
We are ready to prove the stability estimate for the weak solutions to \eqref{eq:weak2}.
\begin{theorem}[Stability]\label{theo.stab}
Under Assumption 1, the solutions $(\bvec{u},\bvec{w},\btens\sigma,p)\in (\bvec{U},\bvec{W},\btens\Sigma,P)$ of problem \eqref{eq:weak2} are such that, for all $t\in(0,\tF]$,
\begin{equation}\label{eq:stabtheorem}
\begin{aligned}
&\norm[]{(2\mu)^{-\frac12}\btens{\rm dev}(\btens\sigma(t))}^2
+\norm[]{\varepsilon^{-\frac12}\btens{\rm skw}(\btens\sigma(t))}^2 +
+\norm[]{R_1^{\frac12}[\bvec{u},\bvec{w}]^{\rm T}(t)}^2
\\
&\quad
+\norm[]{R_2^{\frac12}[{\rm tr}(\btens{\sigma}),p]^{\rm T}(t)}^2
+ \norm[]{\bvec{d}(t)}^2 \hspace{-.5mm}
+\hspace{-1mm} \int_0^t \hspace{-1mm} \norm[]{\diff^{-\frac12}\bvec{w}(s)}^2\ud s
\lesssim C_{\rm pb}^2, 
\end{aligned}
\end{equation}
where the hidden constant does not depend on the dilatation coefficient $\lambda$, the storativity coefficient $s_0$, and the density parameters $\rho_u,\rho_f$, and $\rho_w$.
\end{theorem}
\begin{remark}
    We point out that the stability estimate in \eqref{eq:stabtheorem} is also valid in the degenerate case in which the densities $\rho_u,\rho_f$ and $\rho_w$ are all equal to zero, which corresponds to the quasi-static Biot model. The method presented in this work can be used without modifications. 
\end{remark}
\begin{proof}
The proof is divided into several steps: \textit{(i)} we establish preliminary bounds on the $L^2$-norms of the solid displacement and pressure integrated in time; \textit{(ii)} we infer the control of $\bvec{H}({\rm div})$-norms of the filtration displacement and integrated stress; \textit{(iii)} we deduce an energy balance for system \eqref{eq:weak2}; and \textit{(iv)} we get the conclusion by exploiting the bounds of the previous steps.

\medskip\noindent
\textit{\textbf{Step (i)}: $L^2$-bounds for $\bvec{d}$ and $\tilde{p}$.} Using the discrete inf-sup condition \eqref{eq:infsup_mech}, equation~\eqref{eq:weak3:StressPress} with $q=0$, and the Cauchy-Schwarz inequality, we infer
$$
\begin{aligned}
\beta_{m} \norm[]{\bvec{d}}  &\le \sup_{\btens\tau\in\Hdivmats{\Omega}\cap\bvec{\Sigma}} \frac{\left(\bvec{d},\DIV \btens\tau \right)}{\norm[{\rm div}]{\btens\tau}}\\
&=\sup_{\btens\tau\in\Hdivmats{\Omega}\cap\bvec{\Sigma}}
\frac{\mathcal{L}_3(\btens{\tau}) -
\Big(\frac{\btens{\rm dev}(\btens\sigma)}{2\mu},\btens\tau\Big) -
\Big(\frac{{\rm tr}(\btens\sigma) + \alpha d p}{d^2\kappa}, {\rm tr}(\btens\tau)\Big)
}{\norm[{\rm div}]{\btens\tau}}\\
&\le
\norm[\btens\Sigma^*]{\mathcal{L}_3} 
+\norm[]{(2\mu)^{-1}\btens{\rm dev}(\btens\sigma)}
+ \norm[]{(d\kappa)^{-1}({\rm tr}(\btens\sigma) + \alpha d p)},
\end{aligned}
$$
%with $\norm[\btens\Sigma^*]{\cdot}$ denoting the dual norm $\norm[\btens\Sigma^*]{\mathcal{L}} \coloneq\sup_{\btens\tau\in\bvec{\Sigma}, \norm[{\rm div}]{\btens\tau}=1} \mathcal{L}(\btens{\tau})$.
Hence, using \eqref{eq:databound} and observing that ${d\kappa}^{-1}\le(2\mu)^{-1}\le {\underline{\mu}}^{-1}$, it follows that
\begin{equation}\label{eq:bound.d}
4 C_1 \norm[]{\bvec{d}}^2 \le C_{\rm pb}^2 
+ \norm[]{(2\mu)^{-\frac12}\btens{\rm dev}(\btens\sigma)}^2 
+ \norm[]{R_2^{\frac12}[{\rm tr}(\btens{\sigma}),p]^{\rm T}}^2,
\end{equation}
with $C_1=\underline{\mu}\beta_m^2/12$.
Similarly, 
%letting $\norm[\bvec{W}^*]{\mathcal{L}} \coloneq\sup_{\bvec{z}\in\bvec{W}, \norm[{\rm div}]{\bvec{z}}=1} \mathcal{L}(\bvec{z})$ and 
using the inf-sup condition \eqref{eq:infsup_flux} together with equation~\eqref{eq:weak3:DispFlux} with $\bvec{v}=\bvec{0}$, we obtain
$$
\beta_{f} \norm[]{\tilde{p}} \le \sup_{\bvec{z}\in\bvec{W}} 
\frac{(\tilde{p},\DIV \bvec{z})}{\norm[{\rm div}]{\bvec{z}}}
\le \norm[\bvec{W}^*]{\mathcal{L}_2} +
\norm[]{\rho_f\bvec{u} + \rho_w\bvec{w}}+
\norm[]{\diff^{-1}\tilde{\bvec{w}}}.
$$
Then, noting that for all $t\in(0,\tF]$ it holds $\norm[]{\diff^{-1}\tilde{\bvec{w}}(t)}^2\le t \underline{K}^{-1}
\int_0^t\norm[]{\diff^{-\frac12}\bvec{w}(s)}^2\ud s$, we get
\begin{equation}\label{eq:bound.p}
4 C_2 \norm[]{\tilde{p}}^2 \le C_{\rm pb}^2 
+ \norm[]{R_1^{\frac12}[\bvec{u},\bvec{w}]^{\rm T}}^2
+ \int_0^t\norm[]{\diff^{-\frac12}\bvec{w}(s)}^2\ud s,
\end{equation}
with $C_2>0$ only depending on $\tF$, $\underline{K}$, $\rho_w$, and $\beta_f$.

\medskip\noindent
\textit{\textbf{Step (ii)}: $\bvec{H}({\rm div})$-bounds for $\tilde{\btens{\sigma}}$ and $\tilde{\bvec{w}}$.} 
Testing equation~\eqref{eq:weak3:DispFlux} with $(\bvec{v},\bvec{z})=(\DIV\tilde{\btens\sigma},\bvec{0})$ and equation~\eqref{eq:weak2:StressPress} with
$(\btens\tau,q)=(\bvec{0},\DIV\tilde{\bvec{w}})$ leads to
$$
\begin{aligned}
\norm[]{\DIV\tilde{\btens\sigma}}^2 &= -\mathcal{L}_1(\DIV\tilde{\btens\sigma}) + 
(\rho_u\bvec{u} + \rho_f\bvec{w},\DIV\tilde{\btens\sigma}),\\
\norm[]{\DIV\tilde{\bvec{w}}}^2 &= \mathcal{L}_4(\DIV\tilde{\bvec{w}}) -
\Big(\frac{s_0\kappa+\alpha^2}{\kappa}p,\DIV\tilde{\bvec{w}}\Big)
-\Big(\frac{\alpha {\rm tr}(\btens{\sigma})}{d\kappa}, \DIV\tilde{\bvec{w}}\Big).
\end{aligned}
$$
Then, applying the Cauchy--Schwarz and Young inequalities and using again \eqref{eq:databound}, it is inferred that 
\begin{equation}\label{eq:bound.div}
 4 C_3 (\norm[]{\DIV\tilde{\btens\sigma}}^2 +\norm[]{\DIV\tilde{\bvec{w}}}^2)
 \le C_{\rm pb}^2 
+ \norm[]{R_1^{\frac12}[\bvec{u},\bvec{w}]^{\rm T}}^2 
+ \norm[]{R_2^{\frac12}[{\rm tr}(\btens{\sigma}),p]^{\rm T}}^2,
\end{equation}
with positive constant $C_3$ independent of the quantities $s_0^{-1},\kappa,\rho_f^{-1}$, and $\rho_s^{-1}$.

\medskip\noindent
\textit{\textbf{Step (iii)}: energy estimate.} 
Taking $(\bvec{v},\bvec{z},\btens\tau,q) = 
(\bvec{u},\bvec{w},\btens\sigma,p)$ in \eqref{eq:weak2}, recalling that $R_1$ and $R_2$ are symmetric and positive definite, and using $\partial_t(\varphi^2)=2\varphi\dot{\varphi}$, we have
$$
\begin{aligned}
\partial_t \Big(
\norm[]{R_1^{\frac12}[\bvec{u},\bvec{w}]^{\rm T}}^2
&+\norm[]{R_2^{\frac12}[{\rm tr}(\btens{\sigma}),p]^{\rm T}}^2
+\norm[]{(2\mu)^{-\frac12}\btens{\rm dev}(\btens\sigma)}^2
+\norm[]{\varepsilon^{-\frac12}\btens{\rm skw}(\btens\sigma)}^2
\Big)\\
&\;\qquad+2\norm[]{\diff^{-\frac12}\bvec{w}}^2 = 2 \big(
(\bvec{f},\bvec{u}) + \mathcal{L}_{\bvec{W}}(\bvec{w}) +
\mathcal{L}_{\btens{\Sigma}}(\btens{\sigma}) + (g,p)
\big).
\end{aligned}
$$
Integrating in time the previous identity, assuming without loss of generality that 
$$
\norm[]{R_1^{\frac12}[\bvec{u}_0,\bvec{w}_0]^{\rm T}}^2
+\norm[]{R_2^{\frac12}[{\rm tr}(\btens{\sigma_0}),p_0]^{\rm T}}^2
+\norm[]{(2\mu)^{-\frac12}\btens{\rm dev}(\btens\sigma_0)}^2
\le C_{\rm pb}^2,
$$
and integrating by parts on the right-hand side terms, we infer that for all $t\in(0,\tF]$ 
\begin{equation}\label{eq:bound.energy}
\begin{aligned}
\Big(
&\norm[]{R_1^{\frac12}[\bvec{u},\bvec{w}]^{\rm T}}^2
+\norm[]{R_2^{\frac12}[{\rm tr}(\btens{\sigma}),p]^{\rm T}}^2
+\norm[]{(2\mu)^{-\frac12}\btens{\rm dev}(\btens\sigma)}^2
+\norm[]{\varepsilon^{-\frac12}\btens{\rm skw}(\btens\sigma)}^2
\Big)(t)\\
&+2\int_0^t\norm[]{\diff^{-\frac12}\bvec{w}(s)}^2\ud s \le 
-2 \int_0^t \big(
(\dot{\bvec{f}},\bvec{d}) + \dot{\mathcal{L}}_{\bvec{W}}(\tilde{\bvec{w}}) +
\dot{\mathcal{L}}_{\tilde{\btens{\Sigma}}}(\tilde{\btens{\sigma}}) + (\dot{g},\tilde{p})
\big)(s) \ud s\\
&\qquad\qquad\qquad\qquad\qquad\quad 
+2 \big(
(\bvec{f},\bvec{d}) + \mathcal{L}_{\bvec{W}}(\tilde{\bvec{w}}) +
\mathcal{L}_{\tilde{\btens{\Sigma}}}(\tilde{\btens{\sigma}}) + (g,\tilde{p})
\big)(t) + C_{\rm pb}^2.
\end{aligned}
\end{equation}

\medskip\noindent
\textit{\textbf{Step (iv)}: conclusion.} We collect the estimates \eqref{eq:bound.d}, \eqref{eq:bound.p}, and \eqref{eq:bound.div} to infer that the left-hand side of \eqref{eq:bound.energy} bounds the $L^2$-norms of $\bvec{d},\tilde{p},\DIV\tilde{\btens\sigma}, \DIV\tilde{\bvec w}$, namely
$$
\begin{aligned}
C_1 \norm[]{\bvec{d}(t)}^2 \hspace{-1mm}
+  C_2 \norm[]{\tilde{p}(t)}^2 \hspace{-1mm}
+ C_3 \norm[]{\DIV\tilde{\btens\sigma}(t)}^2 \hspace{-1mm}
+ C_3\norm[]{\DIV\tilde{\bvec{w}}(t)}^2 
\le C_{\rm pb}^2 \hspace{-1mm}+\hspace{-1mm} 
\int_0^t \hspace{-1mm} \norm[]{\diff^{-\frac12}\bvec{w}(s)}^2\ud s
\\
\frac12\big(\norm[]{R_1^{\frac12}[\bvec{u},\bvec{w}]^{\rm T}}^2
+\norm[]{R_2^{\frac12}[{\rm tr}(\btens{\sigma}),p]^{\rm T}}^2
+\norm[]{(2\mu)^{-\frac12}\btens{\rm dev}(\btens\sigma)}^2
+\norm[]{\varepsilon^{-\frac12}\btens{\rm skw}(\btens\sigma)}^2 \big)(t)
\end{aligned}
$$
Moreover, owing to \eqref{eq:databound} and applying the Cauchy--Schwarz and Young inequalities, we bound the second term in the right-hand side of \eqref{eq:bound.energy} as
$$
\begin{aligned}
\big(
(\bvec{f},\bvec{d}) &+ \mathcal{L}_{\bvec{W}}(\tilde{\bvec{w}}) +
\mathcal{L}_{\tilde{\btens{\Sigma}}}(\tilde{\btens{\sigma}}) + (g,\tilde{p})
\big)(t) \\ &\le 
C_4 C_{\rm pb}^2 + \frac{C_1}2 \norm[]{\bvec{d}(t)}^2 \hspace{-1mm}
+  \frac{C_2}2 \norm[]{\tilde{p}(t)}^2 \hspace{-1mm}
+ C_5 \norm[\btens{\Sigma}]{\tilde{\btens\sigma}(t)}^2 \hspace{-1mm}
+ C_5\norm[{\rm div}]{\tilde{\bvec{w}}(t)}^2 
\\ &\le 
C_4 C_{\rm pb}^2 + \frac{C_1}2 \norm[]{\bvec{d}(t)}^2 \hspace{-1mm}
+  \frac{C_2}2 \norm[]{\tilde{p}(t)}^2 \hspace{-1mm}
+ \frac{C_3}2 \norm[]{\DIV\tilde{\btens\sigma}(t)}^2 \hspace{-1mm}
+ \frac{C_3}2 \norm[]{\DIV\tilde{\bvec{w}}(t)}^2
\\ &\quad+ \frac14 \norm[]{(2\mu)^{-\frac12}\btens{\rm dev}(\btens\sigma)}^2
+\frac12 \int_0^t \hspace{-1mm} \norm[]{\diff^{-\frac12}\bvec{w}(s)}^2\ud s,
\end{aligned}
$$
with positive constants $C_4$ and $C_5$ independent of $s_0^{-1},\kappa,\rho_f^{-1}$, and $\rho_s^{-1}$. Plugging the two previous bounds into \eqref{eq:bound.energy}  and using the second inequality in \eqref{eq:databound} yields
$$
\begin{aligned}
&\norm[]{\bvec{d}(t)}^2 \hspace{-.5mm}
+ \norm[]{\tilde{p}(t)}^2 \hspace{-.5mm}
+ \norm[\btens{\Sigma}]{\tilde{\btens\sigma}(t)}^2 \hspace{-.5mm}
+\norm[]{(2\mu)^{-\frac12}\btens{\rm dev}(\btens\sigma(t))}^2
+\norm[]{\varepsilon^{-\frac12}\btens{\rm skw}(\btens\sigma(t))}^2 +
\\
&\quad\norm[]{R_1^{\frac12}[\bvec{u},\bvec{w}]^{\rm T}(t)}^2
+\norm[]{R_2^{\frac12}[{\rm tr}(\btens{\sigma}),p]^{\rm T}(t)}^2
+ \norm[{\rm div}]{\tilde{\bvec{w}}(t)}^2
+\hspace{-1mm} \int_0^t \hspace{-1mm} \norm[]{\diff^{-\frac12}\bvec{w}(s)}^2\ud s
\\
&\quad\lesssim
C_{\rm pb}^2 + \int_0^t C_{\rm pb} (\norm[]{\bvec{d}(s)} \hspace{-.5mm}
+ \norm[]{\tilde{p}(s)} \hspace{-.5mm}
+ \norm[\btens{\Sigma}]{\tilde{\btens\sigma}(s)} \hspace{-.5mm})
+ \norm[{\rm div}]{\tilde{\bvec{w}}(s)}\big) \ud s.
\end{aligned}
$$
The conclusion follows from the Gronwall-type inequality in \cite[Lemma 1]{Lee2016}. 
\end{proof}

%%%%%%%%%%%%%%%%%%%%%%%%%%%%%%%%%%%%%%%%%%%%%%%%%%%%%%%%%%%%%%%%%%%%%%%%%

\section{Discrete setting}\label{sec:discrete.setting}

In this section we present the fully discrete formulation of the weak problem \eqref{eq:weak2}. The spatial discretization is based on coupling two mixed finite element methods, one for the Hellinger--Reissner formulation of linear elasticity and the other for mixed Darcy problems.  
Weakening the symmetry constraint as in \eqref{eq:weak2} allows the use of simpler elements for the stress unknown usually requiring less computational cost. For the time discretization, we employ implicit time-advancing schemes such as $\rm{BDF}$ or $\theta$-method. 
For the sake of simplifying the presentation, in this section we focus on the backward Euler scheme.

\subsection{Space and time meshes and discrete spaces}

The spatial discretization is based on a conforming triangulation $\Th$ of $\Omega$, i.e.\ a set of closed triangles or tetrahedra with union $\overline{\Omega}$ and such that, for any distinct $T_1,T_2\in\Th$, the set $T_1\cap T_2$ is either a common edge, a vertex, the empty set or, if $d=3$, a common face. We denote by $h_T$ the diameter of $T$ and we set
$h\coloneq\max_{T\in\Th}h_T>0$. 
Having in mind a $h$-convergence analysis, we consider a sequence $(\Th)_{h\in\mathcal{H}}$ of refined meshes that is shape-regular in the usual sense of Ciarlet \cite{Ciarlet:02}. 
%For every mesh element $T\in\Th$, we denote by $\Fh[T]$ the set containing the faces that lie on the boundary $\partial T$ of $T$. For any mesh element $T\in\Th$ and each face $F\in\Fh[T]$, $\normal_{TF}$ is the constant unit vector normal to $F$ pointing out of $T$.

The time mesh is obtained subdividing $[0,\tF]$ into $N\in\Natural^*$ uniform subintervals.
We introduce the time step $\tau\coloneq\tF/N$ and the discrete times $t^n\coloneq n\tau$, $n\in\llbracket 0,N\rrbracket$.
For any vector space $V$, all $\varphi\in C^0(V)$, and all $n\in\llbracket 0,N\rrbracket$, we let, for the sake of brevity, $\varphi^n\coloneq\varphi(t^n)$.
We also let, for all $(\varphi^i)_{0\le i\le N}\in V^{N+1}$ and all $1\le n\le N$,
$$
\delta_t^n\varphi\coloneq\frac{\varphi^n-\varphi^{n-1}}{\tau}\in V
$$
denote the backward approximation of the first derivative of $\varphi$ at time $t^n$. 
%For all $n\ge 1$ and $\psi\in L^1(V)$, we define the time average of $\psi$ in $(t^{n-1}, t^n)$ as
%\begin{equation}\label{eq:time.average}
%  \overline\psi^n\coloneq \tau^{-1}\int_{t^{n-1}}^{t^n}\psi(t) {\rm d}t \in V,
%\end{equation}
%with the convention that $\overline\psi^0=0\in V$. 

On the mesh $\Th$, with $h\in\mathcal{H}$, we introduce appropriate finite dimensional spaces 
\begin{equation}\label{eq:conformity}
\begin{aligned}
\btens{\Sigma}_h &\subset \btens{\Sigma},
\qquad \bvec{W}_h \subset\bvec{W},
\qquad \bvec{U}_h \subset \bvec{U},
\qquad P_h \subset P.
\end{aligned}
\end{equation} 
The discrete space $P_h$ (resp. $\bvec{U}_h$) approximating $L^2(\Omega)$ (resp. $\Lvec{\Omega}$) is usually composed of scalar-valued (resp. vector-valued) piecewise discontinuous polynomial functions on $\Th$.
In Section \ref{sec:numerical_results}, the notation $DG_l$ is used  the space spanned by fully discontinuous polynomials of degree at most $l$. 
The $L^2$-projector $\lproj[h]:L^2(X)\to P_h$ is defined such that
\[
  (\lproj[h]q-q, q_h)=0\quad\text{ for all } q_h\in P_h.
\]
When dealing with the space $\bvec{U}_h$, we use the boldface notation $\vlproj[h]$ for the corresponding $L^2$-projector acting component-wise. 
%It is known that the $L^2$-projector $\lproj[h]$ verifies optimal approximation properties (cf. \cite{Dupont.Scott:80}):
%for all $l\in\{0,\ldots,k+1\}$ with $k\ge 0$ denoting the degree of the polynomial space $P_h$, it holds
%\begin{equation}
%  \label{eq:approx.l2proj}
%  \norm[\Omega]{q - \lproj[h] q }
%  \lesssim h^{l} \norm[l,\Omega]{q} \qquad \forall q\in H^l(\Omega).
%\end{equation}

In what follows, we assume that the pair $(\btens{\Sigma}_h,\bvec{U}_h)$ admits a stable element for the weak-symmetry Hellinger--Reissner formulation of linear elasticity in the sense that, there exist a finite dimensional space $\Lambda_h$ such that the triplet $(\btens{\Sigma}_h,\bvec{U}_h, \Lambda_h)$ satisfies the assumptions in \cite[Section 3]{Lee:23} (see, e.g., the Arnold--Falk--Winther element of \cite{Arnold.Falk:07} and the elements presented in \cite{boffi2009reduced}). 
Moreover, the finite element pair $(\bvec{W}_h,P_h)$ is a stable element for the Darcy problem. In Section \ref{sec:numerical_results}, we will consider the standard Raviart--Thomas and Brezzi-Douglas-Marini families.

\subsection{Fully-discrete problem}

The crucial step in the design of the spatial discretization of problem \eqref{eq:weak2} is the selection of the skew-symmetry penalization parameter $\varepsilon$. It has to be small enough so that the solution to the perturbed problem is close to the one of \eqref{eq:biot:weak1}, meaning that the magnitude of the perturbation matches the accuracy of the discretization error. 
At the same time, taking $\varepsilon$ too small may lead to a restrictive constraint to the space $\btens\Sigma_h$ which would compromise the inf-sup stability. In what follows, we choose $\varepsilon$ to be piecewise constant on $\Th$ such that 
$\varepsilon_{|T} =\gamma h_T^r$ for all $T\in\Th$, where $\gamma>0$ is a user-dependent penalty parameter and $r$ denotes the convergence rate of the $L^2$-approximation in $\btens\Sigma_h$ of a sufficiently regular function.

The discretization in space and time of problem \eqref{eq:weak2} consists in finding, for all $1\le n\le N$, the discrete solutions $(\bvec{u}_h^n,\bvec{w}_h^n, \btens{\sigma}_h^n, p_h^n)\in\bvec{U}_h\times\bvec{W}_h\times\bvec{\Sigma}_h\times P_h$ such that
\begin{equation}\label{eq:semidiscrete}
\begin{aligned}
  %\label{eq:discrete:DispFlux}
  \left(R_1\begin{bmatrix}
     \delta_t^n{\bvec{u}_h}    \\
     \delta_t^n{\bvec{w}_h}   
  \end{bmatrix}, 
  \begin{bmatrix}
     \bvec{v}_h    \\
     \bvec{z}_h   
  \end{bmatrix}\right)
  +(\diff^{-1}\bvec{w}_h^n,\bvec{z}_h)
  &- \left(\begin{bmatrix}
     \DIV\btens{\sigma}_h^n   \\
     p_h^n 
  \end{bmatrix},
  \begin{bmatrix}
     \bvec{v}_h   \\
     \DIV\bvec{z}_h  
  \end{bmatrix}\right)  
  =\begin{bmatrix}
     (\bvec{f}^n,\bvec{v}_h)  \\
     \mathcal{L}^n_{\bvec{W}}(\bvec{z}_h)
  \end{bmatrix}
  \\
  %\label{eq:discrete:StressPress}
  \Big(\frac{\btens{\rm dev}(\delta_t^n{\btens\sigma_h})}{2\mu}+
  \frac{\btens{\rm skw}(\delta_t^n{\btens{\sigma_h}})}{\gamma h^r}
  ,\btens\tau_h\Big) +
&\left( R_2\hspace{-.5mm}\begin{bmatrix}
     {\rm tr}(\delta_t^n{\btens\sigma_h}) \\
     \delta_t^n{p_h}
  \end{bmatrix} ,
  \begin{bmatrix}
     {\rm tr}(\btens\tau_h)   \\
     q_h  
  \end{bmatrix}\right)\\
&+ \left(\begin{bmatrix}
     \bvec{u}_h^n   \\
     \DIV\bvec{w}_h^n 
  \end{bmatrix} ,
  \begin{bmatrix}
      \DIV\btens{\tau}_h   \\
     q_h \hspace{-.5mm}
  \end{bmatrix}\right)
  =\hspace{-.5mm}\begin{bmatrix}
     \mathcal{L}^n_{\btens{\Sigma}}(\btens{\tau}_h)\hspace{-.5mm}  \\
     (g^n,q_h)\hspace{-.5mm}
  \end{bmatrix}
\end{aligned}
\end{equation}
for all $(\bvec{v}_h,\bvec{z}_h,\btens\tau_h,q_h)\in \bvec{U}_h\times\bvec{W}_h\times\bvec{\Sigma}_h\times P_h$.
The problem above is completed by the initial conditions  $\btens{\sigma}_{h}^0=\vlproj[h]\btens{\sigma}_{0}$, $\bvec{u}_h^0=\vlproj[h]\bvec{u}_0,\ \bvec{w}_h^0=\vlproj[h]\bvec{w}_0$, and $p_h^0= \lproj[h]p_0$. 
\begin{theorem}
    Suppose that the problem data are as in Assumption 1. 
    Then, for all $1\le n \le N$ problem \eqref{eq:semidiscrete} admits a unique solution $(\bvec{u}_h^n,\bvec{w}_h^n,\btens\sigma_h^n,p_h^n)$ satisfying the a priori estimate 
    \begin{equation}\label{eq:stabdiscrete}
\begin{aligned}
&\norm[]{(2\mu)^{-\frac12}\btens{\rm dev}(\btens\sigma_h^n)}^2
+\norm[]{\varepsilon^{-\frac12}\btens{\rm skw}(\btens\sigma_h^n)}^2 +
+\norm[]{R_1^{\frac12}[\bvec{u}_h^n,\bvec{w}_h^n]^{\rm T}}^2
\\
&\quad
+\norm[]{R_2^{\frac12}[{\rm tr}(\btens{\sigma}_h^n),p_h^n]^{\rm T}(t)}^2
+ \norm[]{\bvec{d}_h^n}^2 \hspace{-.5mm}
+ \sum_{n=1}^N \tau^2 \norm[]{\diff^{-\frac12}\bvec{w}_h^n}^2
\lesssim C_{\rm pb}^2,
\end{aligned}
\end{equation}
with hidden constant not depending on $\lambda$, $s_0$, $\rho_u,\rho_f$, and $\rho_w$, and $C_{\rm pb}$ depending on the problem data as in \eqref{eq:databound}.
\end{theorem}
\begin{proof}
Since \eqref{eq:semidiscrete} corresponds to a set of squared linear systems, uniqueness of solutions implies its existence.  
The uniqueness of solutions can be inferred from \eqref{eq:stabdiscrete}. Indeed, setting to zero the forcing terms, boundary conditions and initial data, makes the right-hand side vanish and, as a result, the only solution would be $(\bvec{u}_h^n,\bvec{w}_h^n,\btens\sigma_h^n,p_h^n)=\bvec{0}$ for all $1\le n\le N$.

A stability estimate as \eqref{eq:stab} holds for the spatial semi-discrete solutions owing to the conformity property \eqref{eq:conformity} of the discrete spaces. Then, in order to prove \eqref{eq:stabdiscrete}, one proceeds as in the proof of Theorem \ref{theo.stab} by replacing integration in time with sum between $1\le n \le N$ and the formula $2\varphi\dot{\varphi}=\partial_t(\varphi^2)$ with its discrete counterpart 
$$2\tau \ \varphi^n \delta_t^n \varphi 
= (\varphi^n)^2 + \tau^2(\delta_t^n \varphi)^2 -(\varphi^{n-1})^2.
\vspace{-6mm}
$$
\end{proof}

For the sake of conciseness, we refer to \cite[Section 3]{Lee:23} for the  convergence analysis with respect to the mesh size $h$ and time step $\tau$ in the case of FEM discretization with weakly-symmetric stress field. 
We postpone the detailed error analysis of the method with skew-symmetry penalization proposed here to a future work, where we also aim to consider more alternatives for the discrete pair $(\btens{\Sigma}_h,\bvec{U}_h)$.

%%%%%%%%%%%%%%%%%%%%%%%%%%%%%%%%%%%%%%%%%%%%%%%%%%%%%%%%%%%%%%%%%%%%%%%%%
\section{Numerical results}
\label{sec:numerical_results}
The aim of this section is to assess the performance of the proposed method in terms of accuracy and robustness. Moreover, it demonstrates its applicability to physically relevant test cases. All computations are performed in \texttt{FEniCS} \cite{Alnaes2015}. We test different choices of discrete spaces for the test presented in these sections. For all the three tests we consider $\btens{\Sigma}_h \times \bvec{U}_h \times P_h = \mathbb{BDM}_{l+1} \times \textbf{DG}_{l} \times DG_{l}$. The difference relies in the space of the approximated filtration velocity. Indeed, in Section~\ref{sec:conv_test} and in the first test of  Section~\ref{sec:rob_test} we consider $\bvec{W}_h = RT_{l}$, while in the second test of Section~\ref{sec:rob_test} and in Section~\ref{sec:wave} we consider it to be $\bvec{W}_h = BDM_{l+1}$. 

\subsection{Convergence test}
\label{sec:conv_test}
We set $\Omega = (0,1)^2$ and consider the following manufactured analytical solution:
$$
    \begin{aligned}
        \bvec{d}(\bvec{x},t) & \ = t^2 \left( \begin{aligned}
            & \sin(2\pi y)(\cos(2\pi x)-1) + \frac{1}{\mu+\lambda}\sin(\pi x)\sin(\pi y) ,\\
            & \sin(2\pi x)(1-\cos(2\pi y)) + \frac{1}{\mu+\lambda} \sin(\pi x)\sin(\pi y)
        \end{aligned} \right), \\
        p(\bvec{x},t) & \ = t^2 \, \sin(\pi x) \sin(\pi y);
    \end{aligned}
$$
Dirichlet boundary conditions and forcing terms are set accordingly. The model coefficients are reported in Table~\ref{tab:params_convtest}. 
\begin{table}[ht]
	\centering 
	\begin{tabular}{r|l c c  r|l c c  r|l c c  r|l }
		$\rho_u \ [\si[per-mode = symbol]{\kilogram \per \metre\cubed}]$ & 1 & & & 
		$\rho_f \ [\si[per-mode = symbol]{\kilogram \per \metre\cubed}]$&  0.25 & & &
		$\rho_w \ [\si[per-mode = symbol]{\kilogram \per \metre\cubed}]$ & 1  & & &
		$\mathbf{K} \ [\si[per-mode = symbol]{\meter\squared \per \pascal\per\second}]$ & $\mathbf{I}$ \\
        $\mu \ [\si[per-mode = symbol]{\pascal}]$ & 1 & & & 
		$\lambda \ [\si[per-mode = symbol]{\pascal}]$&  10 & & &
		$s_0 \ [\si[per-mode = symbol]{\kilogram \per \metre\cubed}]$ & 0.002  & & &
		$\alpha \ [-]$ & 1
	\end{tabular}
	\caption{Convergence tests of Section~\ref{sec:conv_test}: model parameters}
	\label{tab:params_convtest}
\end{table}
We test the convergence of the scheme with respect to the mesh size $h$ assuming that the time step $\tau$ is taken small enough. We consider a sequence of successively refined triangular meshes setting $l=0$. 
In Figure~\ref{fig:conv_p1}, we show the computed errors versus the mesh-size $h$ (loglog scale). 

\begin{figure}[ht]
\begin{subfigure}{1\textwidth}
\centering
\begin{subfigure}[b]{0.4\textwidth}
\begin{tikzpicture}
\begin{axis}[%
width=3.4cm,
height=2.6cm,
at={(0\textwidth,0\textwidth)},
scale only axis,
xmode=log,
xmin=4,
xmax=75,
xminorticks=true,
xlabel={$1/h$},
ymode=log,
ymin=3e-5,
ymax=0.08,
yminorticks=true,
ylabel={$L^2$-errors},
legend style={draw=none,fill=none,legend cell align=left},
legend pos=south west
]
\addplot [color=myred,solid,line width=1.5pt, mark=diamond*,mark options={color=myred}]
  table[row sep=crcr]{
    6.31459315    0.06412655   \\
    12.63640172   0.02309921  \\
    25.14845433   0.01071273  \\
    50.28722371   0.00529773 \\
};

\addplot [color=myblue,solid,line width=1.5pt,mark=square*,mark options={color=myblue}]
  table[row sep=crcr]{
    6.31459315    3.61789688e-03  \\
    12.63640172   5.97728251e-04   \\
    25.14845433   1.50019345e-04   \\
    50.28722371   5.19131421e-05  \\
};

\addplot [color=mygreen,solid,line width=1.5pt,mark=triangle*,mark options={color=mygreen}]
  table[row sep=crcr]{
     6.31459315   0.03036132  \\
    12.63640172   0.00748607 \\
    25.14845433   0.00181166   \\
    50.28722371   0.00045368 \\
};

\addplot [color=myyellow,solid,line width=1.5pt,mark=triangle*,mark options={color=myyellow}]
  table[row sep=crcr]{
     6.31459315   0.04580474   \\
    12.63640172   0.01298907    \\
    25.14845433   0.00341849 \\
    50.28722371   0.00092995     \\
};

\addplot [color=black,solid,line width=0.5pt]
  table[row sep=crcr]{
 25.14845433      0.016 \\
 50.28722371      0.016 \\
 50.28722371      0.008 \\
 25.14845433      0.016 \\ 
};
\node[right, align=left, text=black, font=\footnotesize]
at (axis cs:50.28722371 ,0.012) {1}; 

\addplot [color=black,solid,line width=0.5pt]
  table[row sep=crcr]{
 25.14845433      0.0052 \\
 50.28722371      0.0052 \\
 50.28722371      0.0013 \\
 25.14845433      0.0052 \\ 
};
\node[right, align=left, text=black, font=\footnotesize]
at (axis cs:50.28722371 ,0.003) {2}; 

\end{axis}
\end{tikzpicture}
\end{subfigure}
\hfill
\begin{subfigure}[b]{0.56\textwidth}
\begin{tikzpicture}
\begin{axis}[%
width=3.4cm,
height=2.6cm,
at={(0\textwidth,0\textwidth)},
scale only axis,
xmode=log,
xmin=4,
xmax=75,
xminorticks=true,
xlabel={$1/h$},
ymode=log,
ymin=0.0001,
ymax=0.2,
yminorticks=true,
ylabel={$\text{H}(div)$-errors},
legend style={draw=none,fill=none,legend cell align=left},
legend pos=outer north east
]

\addplot [color=myred,solid,line width=1.5pt, mark=diamond*,mark options={color=myred}]
  table[row sep=crcr]{
    6.31459315    100000\\
    12.63640172   100000 \\
};
\addlegendentry{\scriptsize $\mathbf{u}(\mathbf{x},t)$}

\addplot [color=myyellow,solid,line width=1.5pt,mark=triangle*,mark options={color=myyellow}]
  table[row sep=crcr]{
    6.31459315    100000\\
    12.63640172   100000 \\
};
\addlegendentry{\scriptsize $p(\mathbf{x},t)$}

\addplot [color=myblue,solid,line width=1.5pt,mark=square*,mark options={color=myblue}]
  table[row sep=crcr]{
    6.31459315    0.00746005   \\
    12.63640172   0.00239072  \\
    25.14845433   0.00075618     \\
    50.28722371   0.00025774  \\
};
\addlegendentry{\scriptsize $\mathbf{w}(\mathbf{x},t)$}

\addplot [color=mygreen,solid,line width=1.5pt,mark=triangle*,mark options={color=mygreen}]
  table[row sep=crcr]{
     6.31459315   0.15181387   \\
    12.63640172   0.03655778  \\
    25.14845433   0.00965954   \\
    50.28722371   0.0031448 \\
};
\addlegendentry{\scriptsize $\boldsymbol{\sigma}(\mathbf{x},t)$}

\addplot [color=black,solid,line width=0.5pt]
  table[row sep=crcr]{
 25.14845433      0.016 \\
 50.28722371      0.016 \\
 50.28722371      0.008 \\
 25.14845433      0.016 \\ 
};
\node[right, align=left, text=black, font=\footnotesize]
at (axis cs:50.28722371 ,0.011) {1}; 

\end{axis}
\end{tikzpicture}
\end{subfigure}
\end{subfigure}
\caption{Convergence test of Section~\ref{sec:conv_test}: computed errors in $L^2$-norm (left) and $\text{H}({\rm div})$-norm (right) versus $1/h$ (\textit{log-log} scale). The errors are computed at the final time $T_f$. The polynomial degree of approximation is set to be $l = 0$.}
\label{fig:conv_p1}
\end{figure}
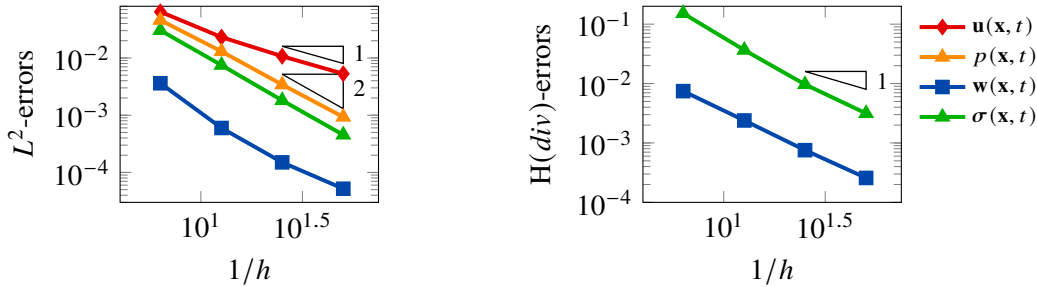

We observe that the $L^2$-error for the velocity field decreases as $h$. 
We observe that, when using the $BDM_{1}$ space for the filtration velocity, the convergence rate of the best $L^2$-approximation is $2$, however due to the coupling with the solid velocity field, its order of accuracy is reduced and it turns out to be suboptimal. Indeed, we observed that, if we off the coupling in the hyperbolic terms (i.e. $\rho_f$, $\rho_w$ = 0), we recover the $h^2$ accuracy for $\bvec{w}$ (cf. Figure~\ref{fig:rob_2}). 
The $L^2$-errors for the stress and the pressure field decrease as $h^2$. For what concerns the $\text{H}({\rm div})$-errors for the stress and the filtration velocity, they decrease better than $h$, that corresponds to the expected accuracy in light of the theory.

It is worth noticing that, with respect to \cite{Lee:23}, where the same discrete spaces are used for the five-fields formulation of the low-frequency poroelasticity model, we gain one order of accuracy for the computed stress field.

%%%%%%%%%%%%%%%%%%%%%%%%%%%%%%%%%%%%%%%%%%%%%%%%%%%%%%%%%%%%%%%%%%%%%%%%%
\subsection{Robustness test}
\label{sec:rob_test}

The aim of this section is to assess the robustness properties of the scheme. To this aim, we consider the same exact solutions, forcing terms, boundary conditions, and discretization parameters of Section~\ref{sec:conv_test} and test the scheme with two different configurations of the model parameters. First, we consider the limit case of quasi-incompressible materials (i.e. $\lambda \rightarrow \infty$) and null specific storage coefficient (i.e. $s_0 = 0$). We consider the same material properties of Table~\ref{tab:params_convtest}, but taking $s_0 = 0$ and $\lambda = 10^{6}$. Second, we consider the case in which $\rho_f =0$ (then, $\rho_w = 0$ as well) which is of interest when we consider the Biot model with the inertial term on displacement only \cite{Bonetti2025}. As in the previous Section, we observe the behavior of $L^2$- and $\text{H}({\rm div})$-errors with respect to the mesh size $h$. In Figure~\ref{fig:rob_1}, Figure~\ref{fig:rob_2} we display the results.

\begin{figure}[ht]
\begin{subfigure}{1\textwidth}
\centering
\begin{subfigure}[b]{0.4\textwidth}
\begin{tikzpicture}
\begin{axis}[%
width=3.4cm,
height=2.6cm,
at={(0\textwidth,0\textwidth)},
scale only axis,
xmode=log,
xmin=4,
xmax=75,
xminorticks=true,
xlabel={$1/h$},
ymode=log,
ymin=3e-5,
ymax=0.08,
yminorticks=true,
ylabel={$L^2$-errors},
legend style={draw=none,fill=none,legend cell align=left},
legend pos=south west
]
\addplot [color=myred,solid,line width=1.5pt, mark=diamond*,mark options={color=myred}]
  table[row sep=crcr]{
    6.31459315    0.06409415  \\
    12.63640172   0.02307243  \\
    25.14845433   0.01070214  \\
    50.28722371   0.00529448 \\
};

\addplot [color=myblue,solid,line width=1.5pt,mark=square*,mark options={color=myblue}]
  table[row sep=crcr]{
    6.31459315    3.51850197e-03  \\
    12.63640172   5.32133759e-04  \\
    25.14845433   1.28338176e-04 \\
    50.28722371   4.78267998e-05   \\
};

\addplot [color=mygreen,solid,line width=1.5pt,mark=triangle*,mark options={color=mygreen}]
  table[row sep=crcr]{
     6.31459315   0.03093045  \\
    12.63640172   0.00761154  \\
    25.14845433   0.00183654  \\
    50.28722371   0.00045388 \\
};

\addplot [color=myyellow,solid,line width=1.5pt,mark=triangle*,mark options={color=myyellow}]
  table[row sep=crcr]{
     6.31459315   0.04974325  \\
    12.63640172   0.01248573    \\
    25.14845433   0.00312425 \\
    50.28722371   0.00078422    \\
};

\addplot [color=black,solid,line width=0.5pt]
  table[row sep=crcr]{
 25.14845433      0.016 \\
 50.28722371      0.016 \\
 50.28722371      0.008 \\
 25.14845433      0.016 \\ 
};
\node[right, align=left, text=black, font=\footnotesize]
at (axis cs:50.28722371 ,0.012) {1}; 

\addplot [color=black,solid,line width=0.5pt]
  table[row sep=crcr]{
 25.14845433      0.0052 \\
 50.28722371      0.0052 \\
 50.28722371      0.0013 \\
 25.14845433      0.0052 \\ 
};
\node[right, align=left, text=black, font=\footnotesize]
at (axis cs:50.28722371 ,0.003) {2}; 

\end{axis}
\end{tikzpicture}
\end{subfigure}
\hfill
\begin{subfigure}[b]{0.56\textwidth}
\begin{tikzpicture}
\begin{axis}[%
width=3.4cm,
height=2.6cm,
at={(0\textwidth,0\textwidth)},
scale only axis,
xmode=log,
xmin=4,
xmax=75,
xminorticks=true,
xlabel={$1/h$},
ymode=log,
ymin=0.0001,
ymax=0.2,
yminorticks=true,
ylabel={$\text{H}(div)$-errors},
legend style={draw=none,fill=none,legend cell align=left},
legend pos=outer north east
]
\addplot [color=myred,solid,line width=1.5pt, mark=diamond*,mark options={color=myred}]
  table[row sep=crcr]{
    6.31459315    10000 \\
    12.63640172   10000 \\
};
\addlegendentry{\scriptsize $\mathbf{u}(\mathbf{x},t)$}

\addplot [color=myyellow,solid,line width=1.5pt,mark=triangle*,mark options={color=myyellow}]
  table[row sep=crcr]{
    6.31459315    10000 \\
    12.63640172   10000 \\
};
\addlegendentry{\scriptsize $p(\mathbf{x},t)$}

\addplot [color=myblue,solid,line width=1.5pt,mark=square*,mark options={color=myblue}]
  table[row sep=crcr]{
    6.31459315    0.00325743 \\
    12.63640172   0.00094051  \\
    25.14845433   0.00038833 \\
    50.28722371    0.00018626  \\
};
\addlegendentry{\scriptsize $\mathbf{w}(\mathbf{x},t)$}

\addplot [color=mygreen,solid,line width=1.5pt,mark=triangle*,mark options={color=mygreen}]
  table[row sep=crcr]{
     6.31459315   0.1512762   \\
    12.63640172   0.03617485   \\
    25.14845433   0.00950198  \\
    50.28722371   0.00310008   \\
};
\addlegendentry{\scriptsize $\boldsymbol{\sigma}(\mathbf{x},t)$}

\addplot [color=black,solid,line width=0.5pt]
  table[row sep=crcr]{
 25.14845433      0.016 \\
 50.28722371      0.016 \\
 50.28722371      0.008 \\
 25.14845433      0.016 \\ 
};
\node[right, align=left, text=black, font=\footnotesize]
at (axis cs:50.28722371 ,0.011) {1};  

\end{axis}
\end{tikzpicture}
\end{subfigure}
\end{subfigure}
\caption{Robustness test of Section~\ref{sec:rob_test} ($s_0 = 0$, $\lambda = 10^6$): computed errors in $L^2$-norm (left) and $\text{H}({\rm div})$-norm (right) versus $1/h$ (\textit{log-log} scale). The errors are computed at the final time $T_f$. The polynomial degree of approximation is set to be $l = 0$.}
\label{fig:rob_1}
\end{figure}
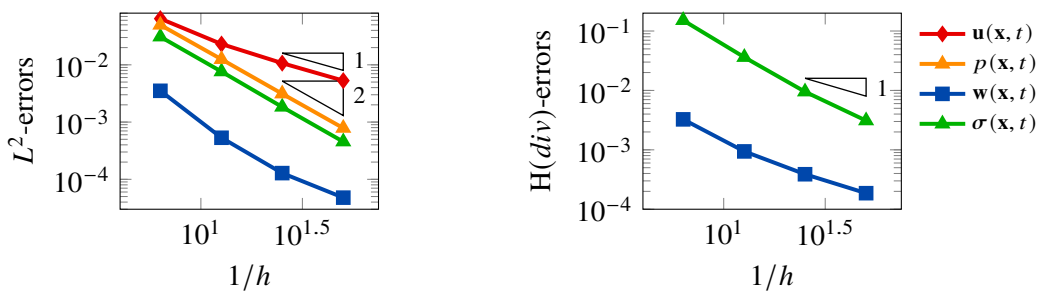

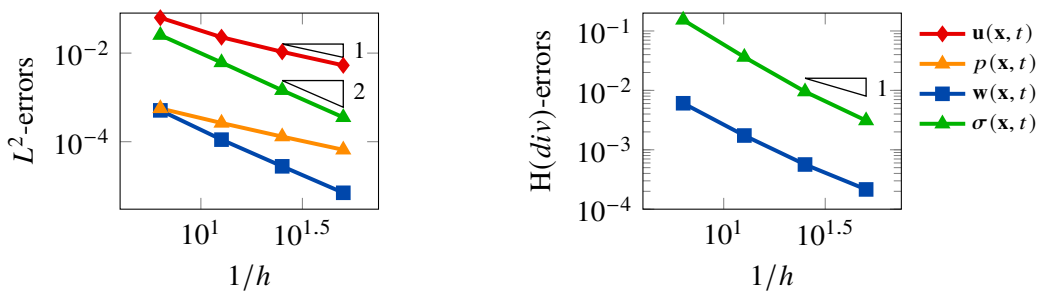
\begin{figure}[ht]
\begin{subfigure}{1\textwidth}
\centering
\begin{subfigure}[b]{0.4\textwidth}
\begin{tikzpicture}
\begin{axis}[%
width=3.4cm,
height=2.6cm,
at={(0\textwidth,0\textwidth)},
scale only axis,
xmode=log,
xmin=4,
xmax=75,
xminorticks=true,
xlabel={$1/h$},
ymode=log,
ymin=3e-6,
ymax=0.08,
yminorticks=true,
ylabel={$L^2$-errors},
legend style={draw=none,fill=none,legend cell align=left},
legend pos=south west
]
\addplot [color=myred,solid,line width=1.5pt, mark=diamond*,mark options={color=myred}]
  table[row sep=crcr]{
    6.31459315    0.06281493   \\
    12.63640172   0.02284122 \\
    25.14845433   0.01068046   \\
    50.28722371   0.00529374  \\
};

\addplot [color=myblue,solid,line width=1.5pt,mark=square*,mark options={color=myblue}]
  table[row sep=crcr]{
    6.31459315    5.09829813e-04 \\
    12.63640172   1.11370465e-04   \\
    25.14845433   2.78327116e-05    \\
    50.28722371   7.03534969e-06   \\
};

\addplot [color=mygreen,solid,line width=1.5pt,mark=triangle*,mark options={color=mygreen}]
  table[row sep=crcr]{
     6.31459315   0.02537697   \\
    12.63640172   0.00613197   \\
    25.14845433   0.00144587  \\
    50.28722371   0.00035435 \\
};

\addplot [color=myyellow,solid,line width=1.5pt,mark=triangle*,mark options={color=myyellow}]
  table[row sep=crcr]{
     6.31459315   5.63477676e-04  \\
    12.63640172   2.64306605e-04  \\
    25.14845433   1.30202555e-04  \\
    50.28722371   6.56016239e-05  \\
};

\addplot [color=black,solid,line width=0.5pt]
  table[row sep=crcr]{
 25.14845433      0.016 \\
 50.28722371      0.016 \\
 50.28722371      0.008 \\
 25.14845433      0.016 \\ 
};
\node[right, align=left, text=black, font=\footnotesize]
at (axis cs:50.28722371 ,0.012) {1}; 

\addplot [color=black,solid,line width=0.5pt]
  table[row sep=crcr]{
 25.14845433      0.0024 \\
 50.28722371      0.0024 \\
 50.28722371      0.0006 \\
 25.14845433      0.0024 \\ 
};
\node[right, align=left, text=black, font=\footnotesize]
at (axis cs:50.28722371 ,0.0014) {2}; 

\end{axis}
\end{tikzpicture}
\end{subfigure}
\hfill
\begin{subfigure}[b]{0.56\textwidth}
\begin{tikzpicture}
\begin{axis}[%
width=3.4cm,
height=2.6cm,
at={(0\textwidth,0\textwidth)},
scale only axis,
xmode=log,
xmin=4,
xmax=75,
xminorticks=true,
xlabel={$1/h$},
ymode=log,
ymin=0.0001,
ymax=0.2,
yminorticks=true,
ylabel={$\text{H}(div)$-errors},
legend style={draw=none,fill=none,legend cell align=left},
legend pos=outer north east
]
\addplot [color=myred,solid,line width=1.5pt, mark=diamond*,mark options={color=myred}]
  table[row sep=crcr]{
    6.31459315    10000 \\
    12.63640172   10000 \\
};
\addlegendentry{\scriptsize $\mathbf{u}(\mathbf{x},t)$}

\addplot [color=myyellow,solid,line width=1.5pt,mark=triangle*,mark options={color=myyellow}]
  table[row sep=crcr]{
    6.31459315    10000 \\
    12.63640172   10000 \\
};
\addlegendentry{\scriptsize $p(\mathbf{x},t)$}

\addplot [color=myblue,solid,line width=1.5pt,mark=square*,mark options={color=myblue}]
  table[row sep=crcr]{
    6.31459315    0.00606919  \\
    12.63640172   0.00173802   \\
    25.14845433   0.00056518  \\
    50.28722371   0.00021498 \\
};
\addlegendentry{\scriptsize $\mathbf{w}(\mathbf{x},t)$}

\addplot [color=mygreen,solid,line width=1.5pt,mark=triangle*,mark options={color=mygreen}]
  table[row sep=crcr]{
     6.31459315   0.15328625   \\
    12.63640172   0.03628422    \\
    25.14845433   0.00951534   \\
    50.28722371   0.00309927    \\
};
\addlegendentry{\scriptsize $\boldsymbol{\sigma}(\mathbf{x},t)$}

\addplot [color=black,solid,line width=0.5pt]
  table[row sep=crcr]{
 25.14845433      0.016 \\
 50.28722371      0.016 \\
 50.28722371      0.008 \\
 25.14845433      0.016 \\ 
};
\node[right, align=left, text=black, font=\footnotesize]
at (axis cs:50.28722371 ,0.011) {1};  

\end{axis}
\end{tikzpicture}
\end{subfigure}
\end{subfigure}
\caption{Robustness test of Section~\ref{sec:rob_test} ($\rho_f = \rho_w = 0$): computed errors in $L^2$-norm (left) and $\text{H}({\rm div})$-norm (right) versus $1/h$ (\textit{log-log} scale). The errors are computed at the final time $T_f$. The polynomial degree of approximation is set to be $l = 0$.}
\label{fig:rob_2}
\end{figure}

For what concerns Figure~\ref{fig:rob_1}, we observe that both the $L^2$- and $\text{H}({\rm div})$-errors respect the estimated order of accuracy and are in agreement with the results found in Section~\ref{sec:conv_test}. Moreover, we observe that the absolute values of the errors are almost the same of the ones observed in Section~\ref{sec:conv_test}. Then, we can conclude that the proposed scheme is robust with respect to the quasi-incompressible case and null specific storage coefficient.

In Figure~\ref{fig:rob_2} we observe that the $\text{H}({\rm div})$-errors behave as in the convergence test and in the first robustness test. Regarding the $L^2$-errors, we observe the exact order of accuracy predicted by the theory. It is interesting to notice that, with respect to Figure~\ref{fig:conv_p1} and Figure~\ref{fig:rob_1}, the $L^2$-error of the pressure field decays as $h$, while the $L^2$-error of the filtration velocity field decreases as $h^2$. In this case, as we neglect the coupling between the solid velocity and the filtration velocity, we are able to recover the second-order accuracy for $\bvec{w}$. We remark again that, to obtain this result, we need to use $BDM_{1}$ as approximation space of $\bvec{w}$ instead of $RT_0$.

%%%%%%%%%%%%%%%%%%%%%%%%%%%%%%%%%%%%%%%%%%%%%%%%%%%%%%%%%%%%%%%%%%%%%%%%%
\subsection{Wave propagation in a poroleastic medium}
\label{sec:wave}
In this section we test the proposed method in a physically-sound test case inspired by \cite{Bonetti2023, Meddahi:25}. We investigate a wave propagation problem in a homogeneous and isotropic poroelastic medium. The computational domain is $\Omega = (0,4800)\times(0,4800) \si{\meter\squared}$ and we place an explosive source \cite{Meddahi:25, Morency2008} at position $\mathbf{x} = (2400, 2400)\si{\meter}$
$$
\bvec{f}(\bvec{x},t) = S(t) \left\{
\begin{aligned}
    & \left(1 - \frac{\| \bvec{r} \|^2}{4h^2} \right)\frac{\bvec{r}}{\|\bvec{r}\|} \qquad \text{if} \, \|\bvec{r}\| < 2h \\
    & 0 \qquad \text{otherwise}
\end{aligned}
\right.
$$
where $S(t) = (1 - 2\pi^2 f_0^2 (t-t_0)^2)\exp(-\pi^2 f_0^2(t-t_0)^2)$, $f_0$ is the peak frequency, $t_0$ is the time shift parameter, $h$ corresponds to the mesh size, and the vector $\bvec{r} = (x-2400, y-2400)^T$ denotes the distance from the source position. This choice of the forcing term is often used in the context of earthquakes and induces a smooth, radially symmetric force distribution around the source point and the temporal function $S(t)$ generates a Ricker wavelet.

The homogeneous isotropic medium is characterized by the following physical parameters, cf. Table~\ref{tab:params_wave}.
\begin{table}[ht]
	\centering 
	\begin{tabular}{l|l c c  l|l c c  l|l}
		$\rho_u \ [\si[per-mode = symbol]{\kilogram \per \metre\cubed}]$ & 1700 & & & 
		$\rho_f \ [\si[per-mode = symbol]{\kilogram \per \metre\cubed}]$&  950 & & &
		$\rho_w \ [\si[per-mode = symbol]{\kilogram \per \metre\cubed}]$ & 4750  \\
		$\mathbf{K} \ [\si[per-mode = symbol]{\meter\squared \per \pascal\per\second}]$ & $\num[exponent-product=\ensuremath{\cdot}]{6.6667e-10}\mathbf{I}$ & & & 
        $\mu \ [\si[per-mode = symbol]{\pascal}]$ & \num[exponent-product=\ensuremath{\cdot}]{7.2073e+9} & & & 
		$\lambda \ [\si[per-mode = symbol]{\pascal}]$ & \num[exponent-product=\ensuremath{\cdot}]{4.3738e+9} \\
		$s_0 \ [\si[per-mode = symbol]{\kilogram \per \metre\cubed}]$ &  \num[exponent-product=\ensuremath{\cdot}]{1.462e-10}  & & &
		$\alpha \ [-]$ & 0.029
	\end{tabular}
	\caption{Wave propagation test of Section~\ref{sec:wave}: pooelastic medium parameters}
	\label{tab:params_wave}
\end{table}
Finally, the problem is closed by homogeneous Dirichlet boundary conditions. Since we are using Dirichlet boundary conditions, in our simulation we take into account the results in which the wavefront has not yet touched the boundary of the domain. The mesh is made of $N = 120000$ elements ($h\sim51\si{\meter}$), the polynomial degree of approximation is set to be $l = 1$, and the time-step parameter for the backward Euler time marching scheme is set to be $\Delta t = 0.005$. Last, the final time of simulation is $T_f = 1\si{\second}$.
In the following, we denote by $\bvec{u}_h$ the solid velocity (i.e. $\dot{\bvec{d}}_h$), by $u_{h,y}$ its vertical component, and by $p_{h}$ the pressure field.
We report in Figure~\ref{fig:TestFisico_v}, Figure~\ref{fig:TestFisico_vy}, and Figure~\ref{fig:TestFisico_p}, the computed quantities $|\bvec{u}_h|$, $u_{h,y}$, and $p_{h}$ at selected time instants, respectively.
\begin{figure}[ht]
\begin{subfigure}[b]{.33\textwidth}
    \centering
    \includegraphics[width=1\textwidth]{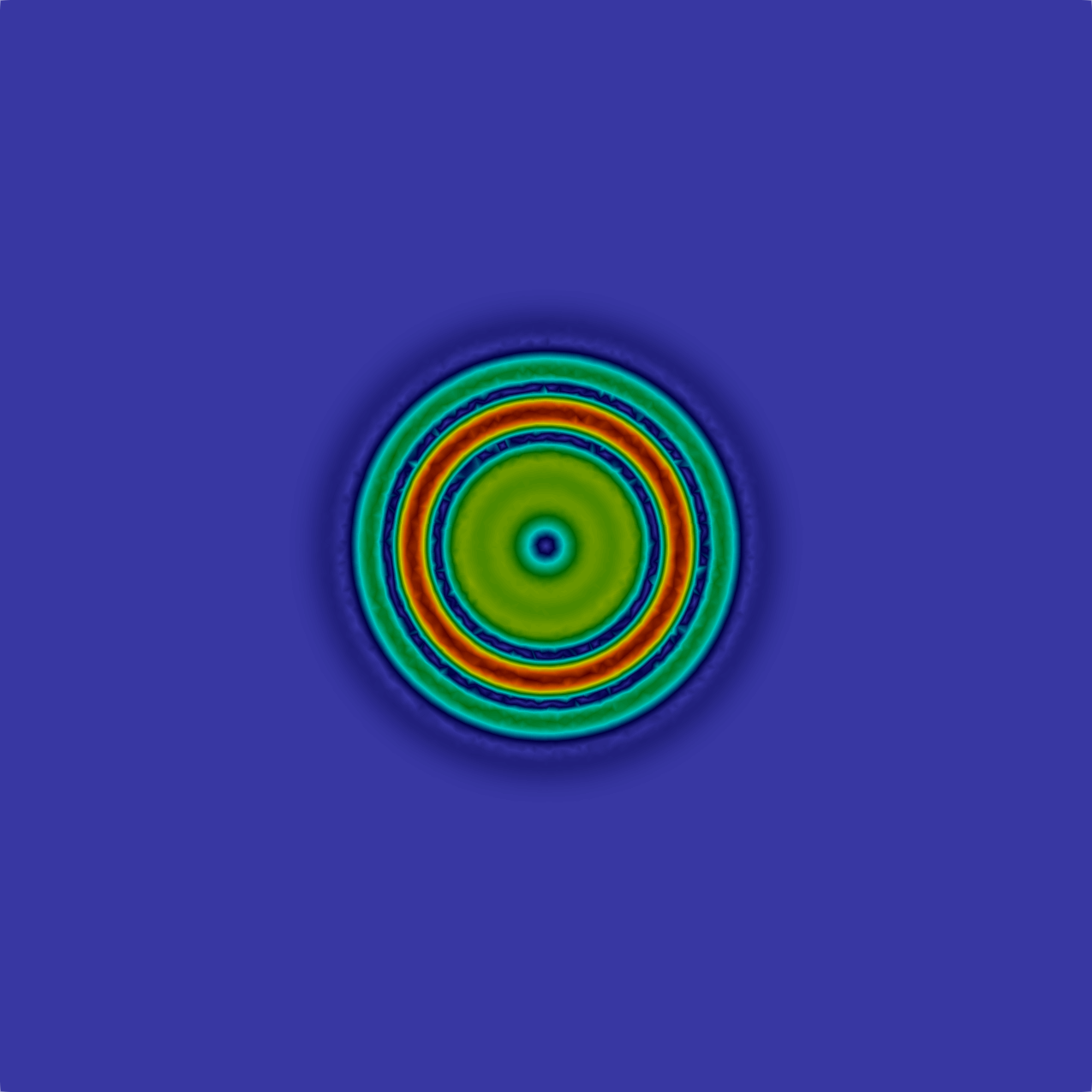}
    \label{fig:TestFisico_vy_08}
\end{subfigure}
\begin{subfigure}[b]{.33\textwidth}
    \centering
    \includegraphics[width=1\textwidth]{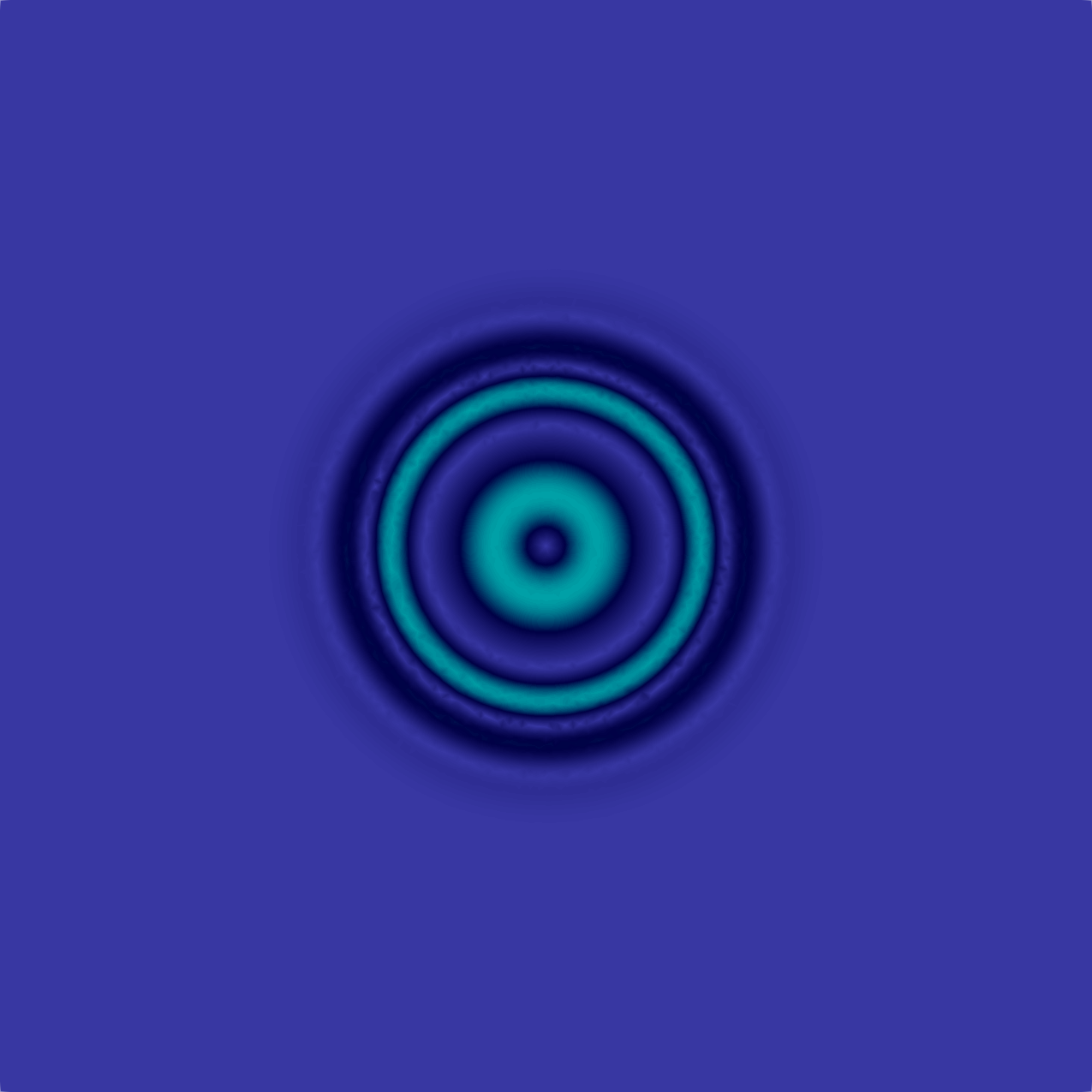}
    \label{fig:TestFisico_vy_09}
\end{subfigure}
\begin{subfigure}[b]{.33\textwidth}
    \centering
    \includegraphics[width=1\textwidth]{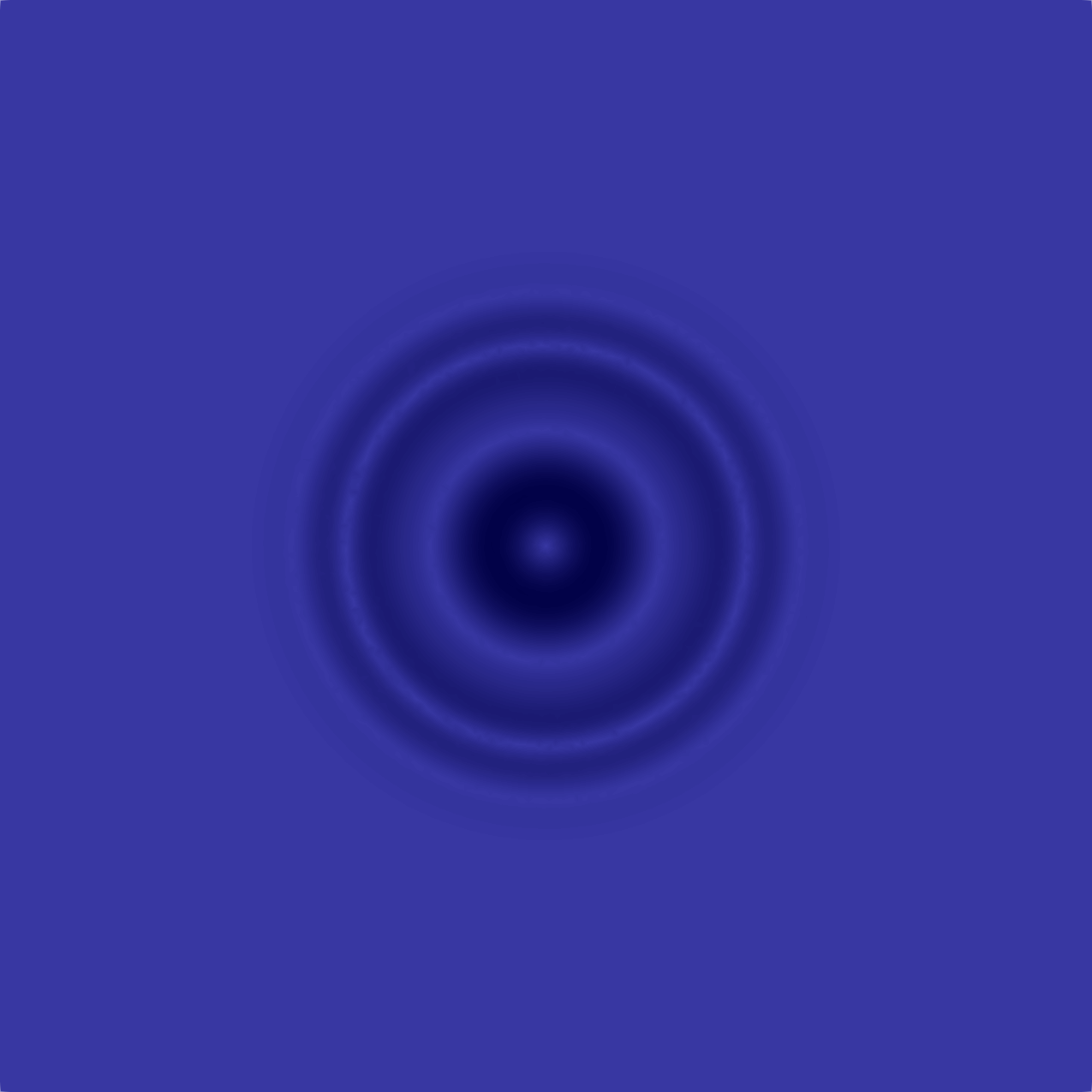}
    \label{fig:TestFisico_vy_1}
\end{subfigure}

\vspace{-1.1cm}
\hspace{0.1cm}
\includegraphics[width=0.3\textwidth]{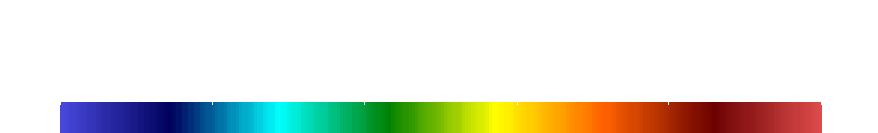}

\caption{Wave propagation in poroelastic medium: modulus of the computed velocity field $|\bvec{u}_{h}|$ at the time instants $t=0.8 \si{\second}$ (left), $t=0.9 \si{\second}$ (center), $t=1 \si{\second}$ (right).}
\label{fig:TestFisico_v}
\end{figure}

\begin{figure}[ht]
\begin{subfigure}[b]{.33\textwidth}
    \centering
    \includegraphics[width=1\textwidth]{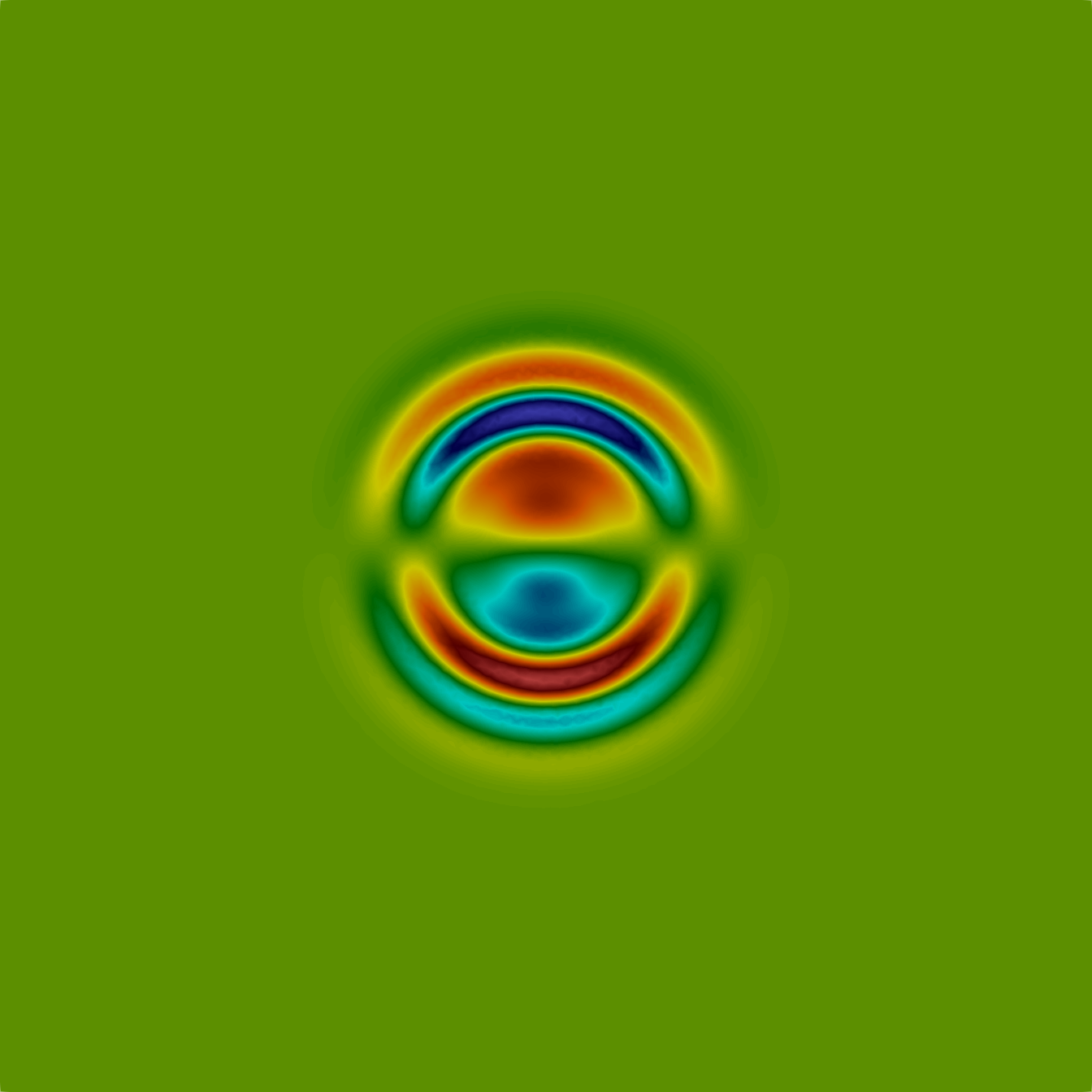}
    \label{fig:TestFisico_vy_08b}
\end{subfigure}
\begin{subfigure}[b]{.33\textwidth}
    \centering
    \includegraphics[width=1\textwidth]{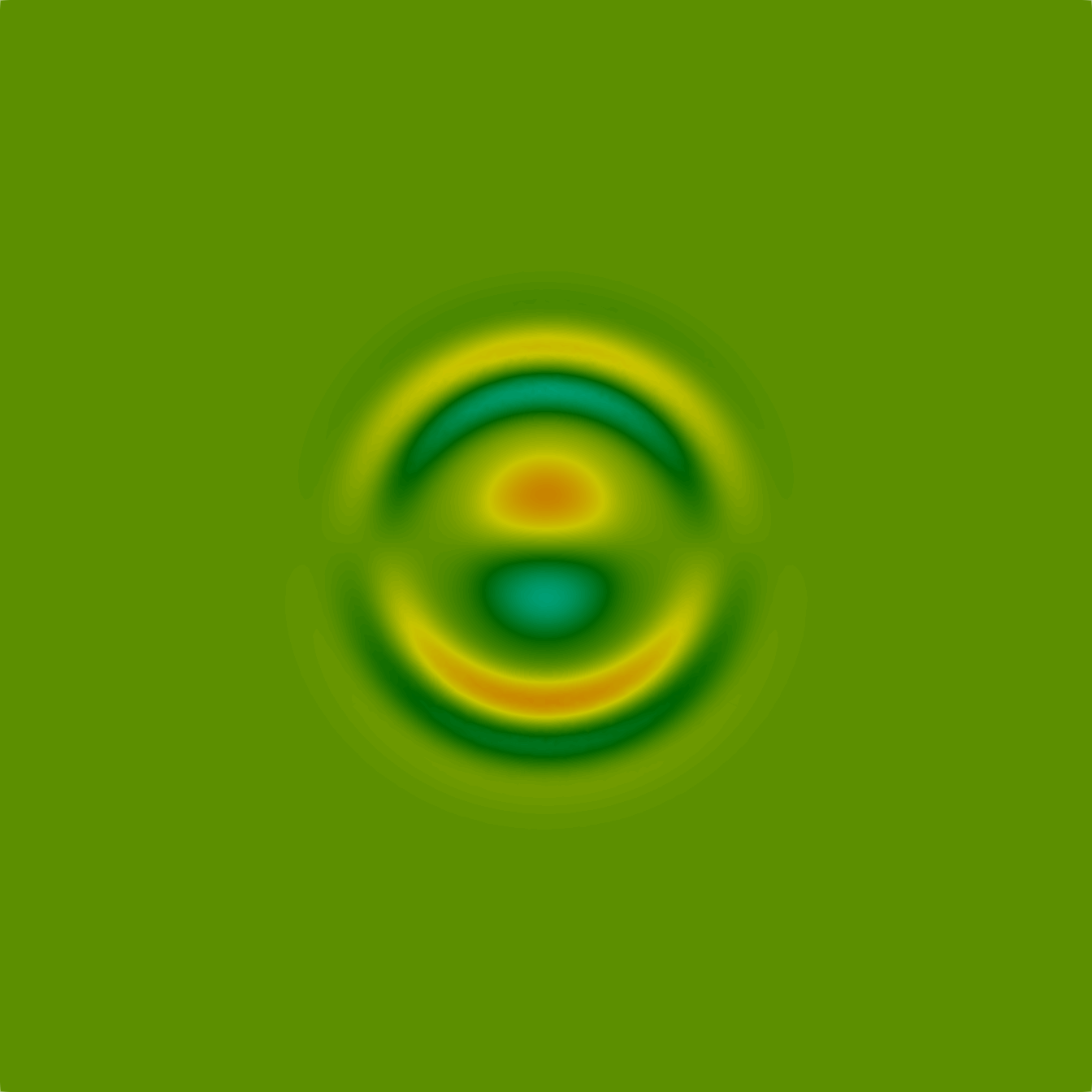}
    \label{fig:TestFisico_vy_09b}
\end{subfigure}
\begin{subfigure}[b]{.33\textwidth}
    \centering
    \includegraphics[width=1\textwidth]{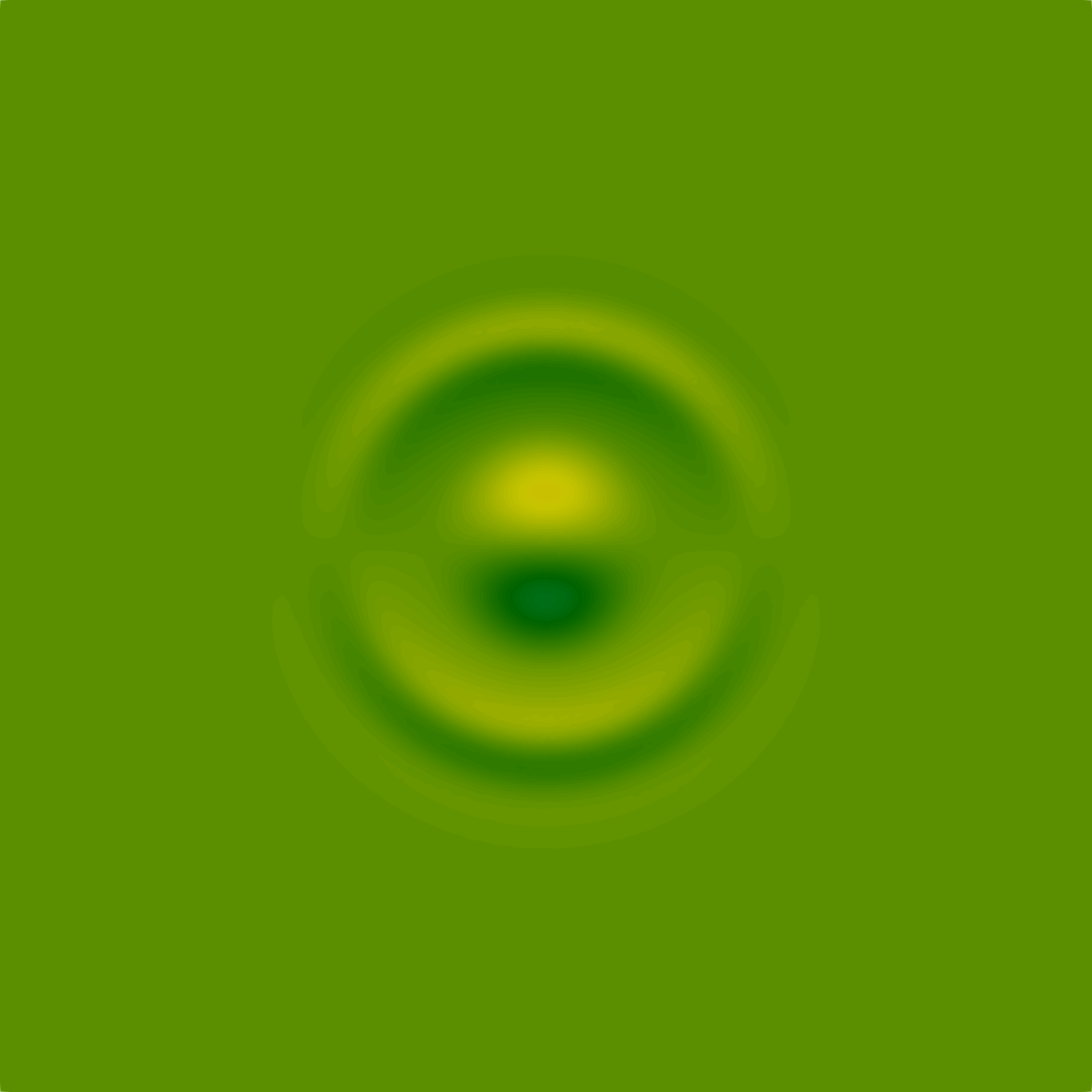}
    \label{fig:TestFisico_vy_1b}
\end{subfigure}

\vspace{-1.1cm}
\hspace{0.1cm}
\includegraphics[width=0.3\textwidth]{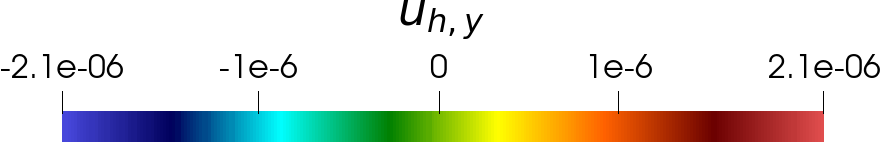}

\caption{Wave propagation in poroelastic medium: computed vertical component of the velocity $u_{h,y}$ at the time instants $t=0.8 \si{\second}$ (left), $t=0.9 \si{\second}$ (center), $t=1 \si{\second}$ (right).}
\label{fig:TestFisico_vy}
\end{figure}

\begin{figure}[ht]
\begin{subfigure}[b]{.33\textwidth}
    \centering
    \includegraphics[width=1\textwidth]{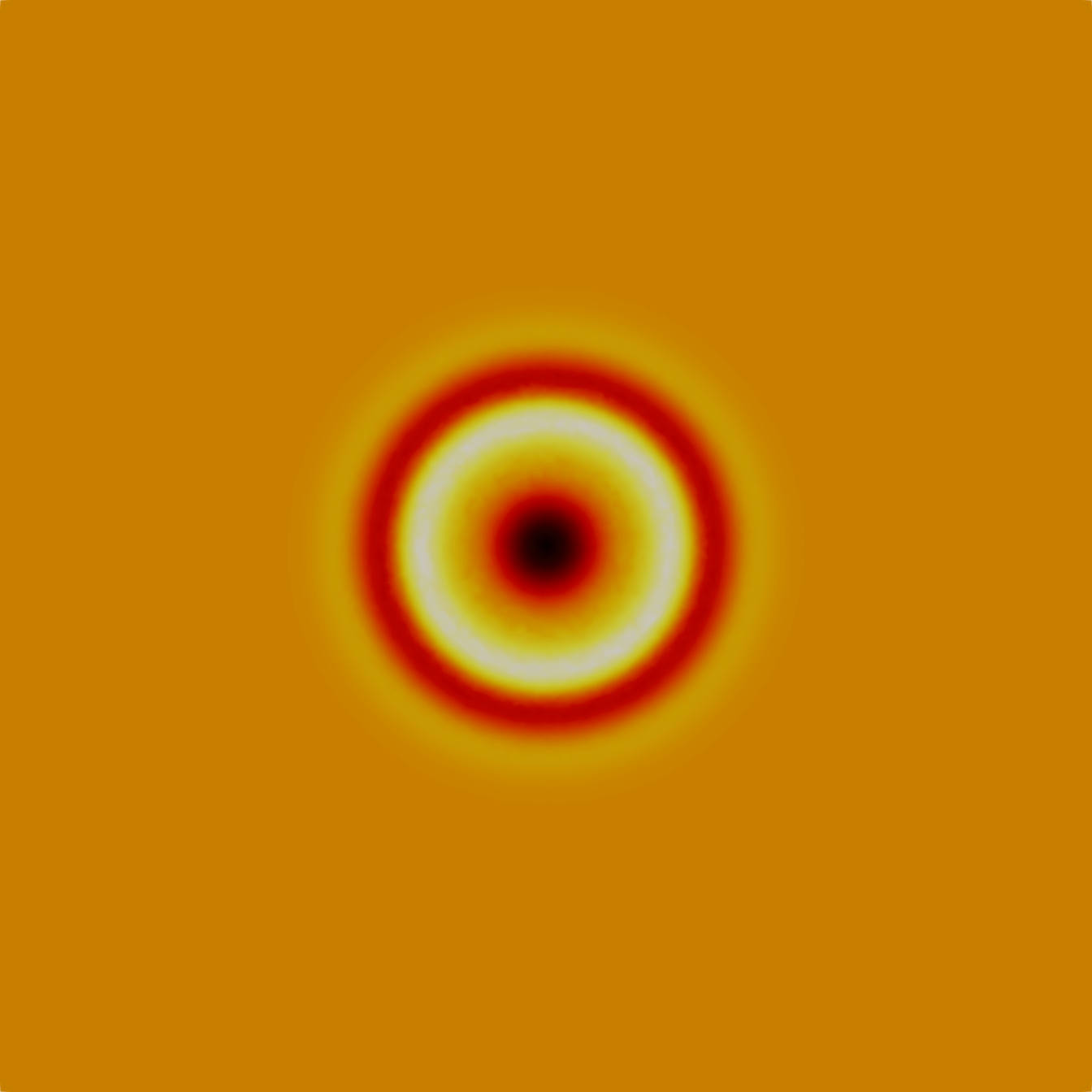}
    \label{fig:TestFisico_p_08}
\end{subfigure}
\begin{subfigure}[b]{.33\textwidth}
    \centering
    \includegraphics[width=1\textwidth]{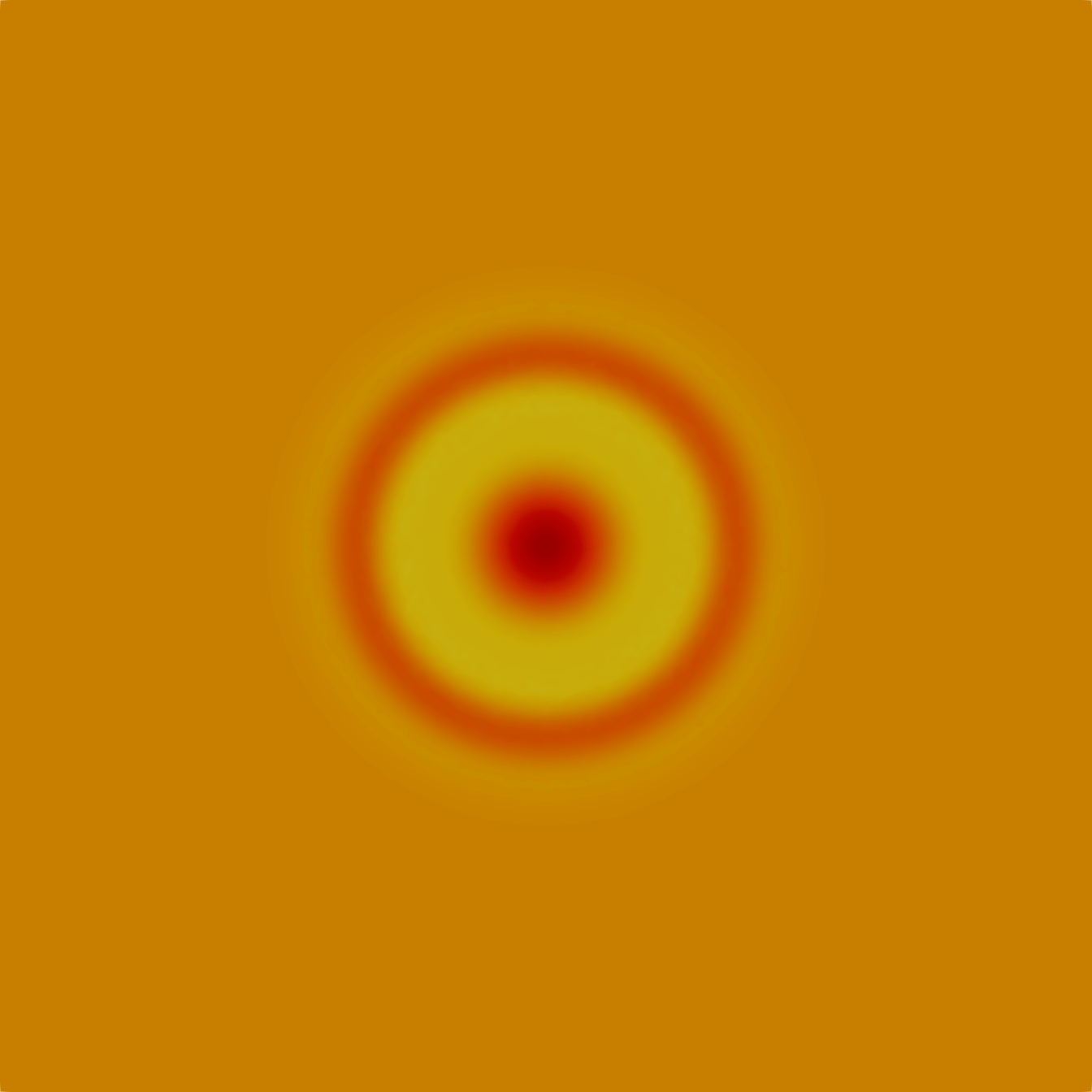}
    \label{fig:TestFisico_p_09}
\end{subfigure}
\begin{subfigure}[b]{.33\textwidth}
    \centering
    \includegraphics[width=1\textwidth]{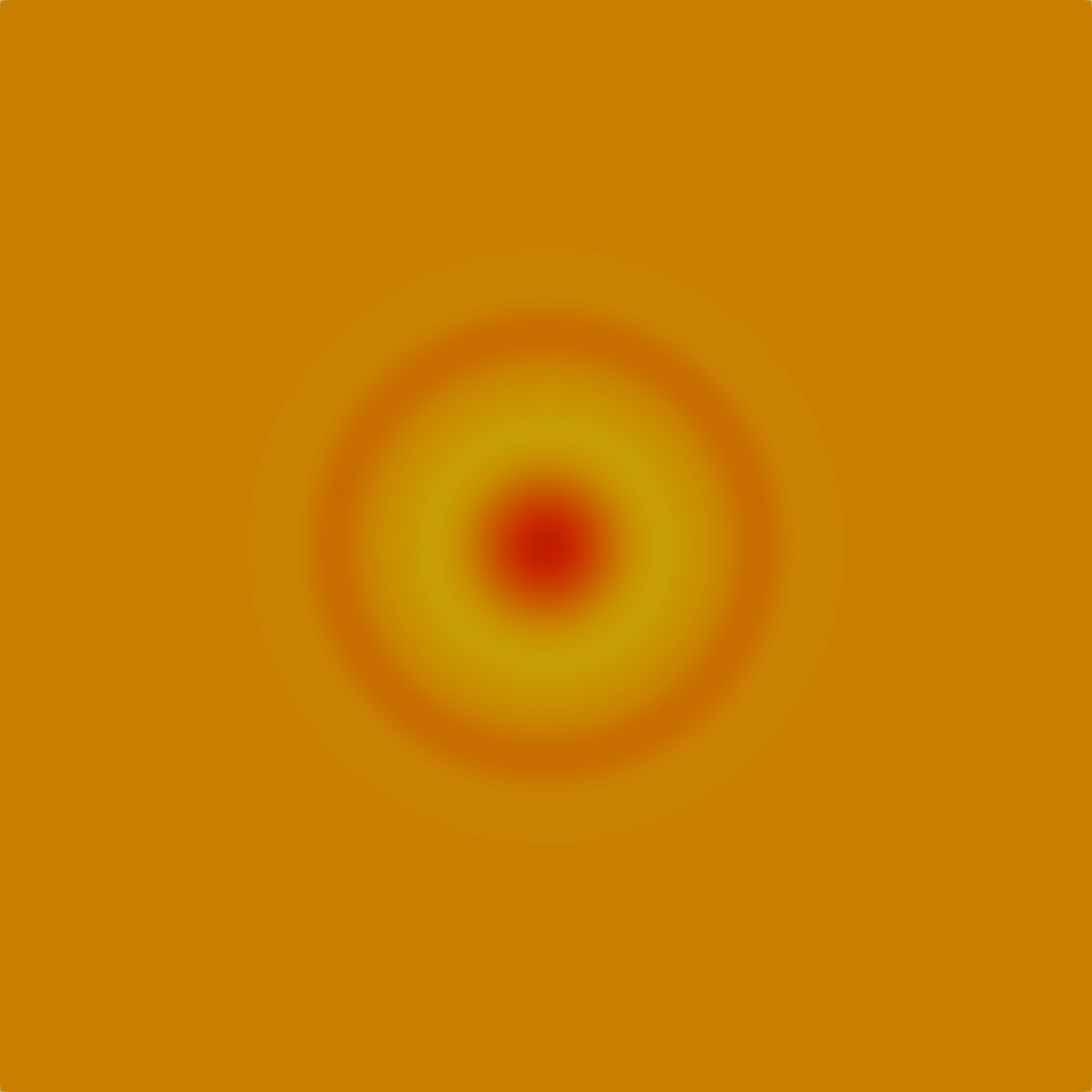}
    \label{fig:TestFisico_p_1}
\end{subfigure}

\vspace{-1.1cm}
\hspace{0.1cm}
\includegraphics[width=0.3\textwidth]{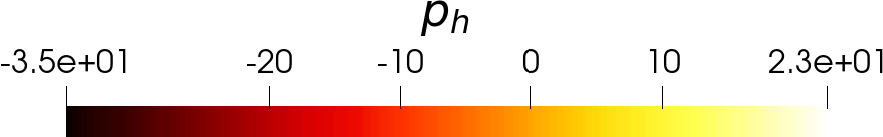}

\caption{Wave propagation in poroelastic medium: computed pressure field $p_{h}$ at the time instants $t=0.8 \si{\second}$ (left), $t=0.9 \si{\second}$ (center), $t=1 \si{\second}$ (right).}
\label{fig:TestFisico_p}
\end{figure}

From the results of Figure~\ref{fig:TestFisico_v} we notice a 
symmetric wavefront that detaches from the center of the domain; this is due to the homogeneity of the poroelastic material in which it propagates. We can see that the axes of symmetry of our wavefront are $x$- and $y$-axis of the square domain. This behavior is correct and is due to the form of the forcing term we are imposing. From the results of Figure~\ref{fig:TestFisico_vy}, we can observe the fast $P$-wave that propagates first in the domain, and we can observe also the presence of the slow $P$-wave and of the $E$-wave. Due to the choice of the forcing terms, the shear waves are not observed in this simulation. Last, by looking at Figure~\ref{fig:TestFisico_p}, we can observe that -- as expected -- the computed pressure field follows the path of the velocity field.
In conclusion, we can see a good agreement between our results and the ones presented in \cite{Antonietti2022, Bonetti2023, Meddahi:25}.

\medskip\medskip\noindent
\textbf{Acknowledgements and competing Interests}
This research has been partially funded by the European Union (ERC, NEMESIS, project number 101115663). Views and opinions expressed are however those of the author(s) only and do not necessarily reflect those of the European Union or the European Research Council Executive Agency. The research is part of the activities of \textit{Dipartimento di Eccellenza 2023-2027}, funded by MUR, Italy.
SB and MB are members of INdAM-GNCS and have been partially supported by the INdAM-GNCS project CUP E53C24001950001.
PV's work was done with the support of the Chilean grant ANID FONDECYT No. 11251691.

%\eject
\bibliographystyle{plain}
\bibliography{arxiv}
\end{document}